\documentclass[11pt,leqno]{article}
\usepackage{amsmath,amsfonts,amscd,amsthm}
\begin{document}

\newcommand{\se}{\setcounter{equation}{0}}
\def\theequation{\thesection.\arabic{equation}}

\newtheorem{theorem}{Theorem}[section]
\newtheorem{cdf}{Corollary}[section]
\newtheorem{lemma}{Lemma}[section]
\newtheorem{remark}{Remark}[section]

\newcommand{\R}{\mathbb{R}}
\newcommand{\e}{{\bf e}}
\newcommand{\E}{{\bf E}}
\newcommand{\f}{{\bf f}}
\newcommand{\he}{\hat{\e}}
\newcommand{\hf}{\hat{\f}}
\newcommand{\te}{\tilde{\e}}

\newcommand{\bu}{{\bf u}}
\newcommand{\hbu}{\hat{\bu}}
\newcommand{\bv}{{\bf v}}
\newcommand{\bw}{{\bf w}}
\newcommand{\by}{{\bf y}}
\newcommand{\bz}{{\bf z}}
\newcommand{\bH}{{\bf H}}
\newcommand{\bJ}{{\bf J}}
\newcommand{\bL}{{\bf L}}

\newcommand{\bchi}{\mbox{\boldmath $\chi$}}
\newcommand{\bta}{\mbox{\boldmath $\eta$}}
\newcommand{\tbta}{{\tilde{\bta}}}
\newcommand{\bphi}{\mbox{\boldmath $\phi$}}
\newcommand{\bth}{\mbox{\boldmath $\theta$}}
\newcommand{\bxi}{\mbox{\boldmath $\xi$}}
\newcommand{\bzt}{\mbox{\boldmath $\zeta$}}

\newcommand{\hta}{\hat{\bta}}
\newcommand{\hth}{\hat{\bth}}
\newcommand{\hxi}{\hat{\bxi}}
\newcommand{\hzt}{\hat{\bzt}}

\newcommand{\td}{\tilde{\Delta}}

\newcommand{\bbu}{\bar{\bu}}
\newcommand{\bby}{\bar{\by}}
\newcommand{\bbz}{\bar{\bz}}

\title
{ \large\bf A two-level finite element method for time-dependent incompressible Navier-Stokes equations with non-smooth initial data}

\author{Deepjyoti Goswami\footnote{
Department of Mathematics, Universidade Federal do Paran\'a, Centro Polit\'ecnico, Curitiba,
Cx.P: 19081, CEP: 81531-990, PR, Brazil}~
 and Pedro D. Dam\'azio\footnote{
 Department of Mathematics, Universidade Federal do Paran\'a, Centro Polit\'ecnico, Curitiba,
 Cx.P: 19081, CEP: 81531-990, PR, Brazil}
}
\date{}
\maketitle

\begin{abstract}
In this article, we analyze a two-level finite element method for the two dimensional time-dependent incompressible Navier-Stokes equations with non-smooth initial data. It involves solving the non-linear Navier-Stokes problem on a coarse grid of size $H$ and solving a Stokes problem on a fine grid of size $h,~h<<H$. This method gives optimal convergence for velocity in $H^1$-norm and for pressure in $L^2$-norm. The analysis takes in to account the loss of regularity of the solution at $t=0$ of the Navier-Stokes equations.
\end{abstract}

\vspace{0.30cm} 

\vspace{.2in}
\section{Introduction}
In this article, we consider a two-level semi-discrete finite element approximation to the 
two dimensional time-dependent incompressible Navier-Stokes equations
\begin{eqnarray}\label{nse}
 ~~\frac {\partial \bu}{\partial t}+\bu\cdot\nabla\bu-\nu\Delta\bu+\nabla p=\f(x,t),~~x\in \Omega,~t>0
\end{eqnarray}
with incompressibility condition
\begin{eqnarray}\label{ic}
 \nabla \cdot \bu=0,~~~x\in \Omega,~t>0,
\end{eqnarray}
and initial and boundary conditions
\begin{eqnarray}\label{ibc}
 \bu(x,0)=\bu_0~~\mbox {in}~\Omega,~~~\bu=0,~~\mbox {on}~\partial\Omega,~t\ge 0.
\end{eqnarray}
Here, $\Omega$ is a bounded domain in $\R^2$ with boundary $\partial \Omega$ and 
$\nu >0$ is the viscosity. $\bu$ and $p$ are the velocity field and pressure, respectively. And $\f$ is a given force field.

Two-level or two-grid methods are well-established and efficient methods for solving non-linear partial differential equations. Due to high computational cost of solving a non-linear problem on a fine grid, we solve the original problem on a coarse grid and update the solution by solving a linearized problem on a fine grid. In other words, in the first step, we discretize the non-linear PDE on a coarse mesh, of mesh-size $H$ and compute an approximate solution, say, $\bu_H$. Then, in the second step, we formulate a linearized problem, out of the original one, using $\bu_H$ and discretize it on a fine mesh, of mesh-size $h,~h<<H$, thereby, compute an approximate solution, say, $\bu^h$. With appropriate $h,H$, we obtain same order of convergence of the error $\bu-\bu^h$, as that of the error obtained by semi-discrete finite element Galerkin approximation on the find grid; but with far less computational cost, since, instead of solving a large non-linear system, we solve a small non-linear system and a large linear system.

In this article, we study the following two-level finite element approximation for the problem (\ref{nse})-(\ref{ibc}): First, we compute a semi-discrete Galerkin finite element approximations $(\bu_H,p_H)$, over a coarse mesh of mesh-size $H$. Then, we use the approximation $\bu_H$ to compute a semi-discrete Galerkin finite element approximations $(\bu^h,p^h)$ of the following Stokes problem:
\begin{align}\label{second}
\bu_t-\nu\Delta\bu+\bu_H\cdot\nabla\bu_H+\nabla p =\f,~~\nabla \cdot \bu =0,
~~\mbox{in}~\Omega
\end{align}
over a fine mesh of mesh-size $h<<H$.

The above algorithm is nothing new and in fact, this and similar algorithms have been studied on numerous occasions. But to the best of our knowledge, no study has been done for non-smooth initial data. Our main aim in this work is to do error analysis of the above mentioned two-level method, under non-smooth initial data.

Two-grid method was first introduced by Xu \cite{Xu94,Xu96} for semi-linear elliptic problems and by Layton {\it et.al} \cite{Lay93,LL95,LT98} for steady Navier-Stokes equations. It was carried out for time dependent Navier-Stokes by Girault and Raviart \cite{GL01} for semi-discrete case. The method may vary depending on the algorithm, by formulating different  linearize problem, to solve in the second step; like, in the case of Navier-Stokes, one can chose a Stokes problem or an Oseen problem or a Newton step to solve on the fine mesh. Several works in this direction, involving both semi-discrete and fully discrete analysis, can be found in \cite{AGS09, AS08, He04, HL11, HMR04, LH10, LH10a, LHL12, Ol99} and references therein.

The two-level methods are similar to non-linear Galerkin methods, post-processed and dynamical post-processed methods, in the sense that, all these method try to control the computational cost or efficiency by controlling the non-linear term, implementing similar ideas: solve the non-linear problem on a coarse grid and use that solution to solve a linearized problem on a fine grid. A discussion on this can be found out in \cite{FGN12}.

But in all these methods, including two-level, very few articles can be found to carry out the analysis, taking into account the loss of regularity of the solution of the Navier-Stokes equations at the initial time. Recently, two articles \cite{FGN12, FGN12a} have taken this into account. For example, both these articles assume no more than second-order spatial derivative of the velocity bounded in $\bL^2$, up to initial time $t=0$. Under realistic assumptions on the given data, the solution of the Navier-Stokes equations suffers from lack of regularity at $t=0$. For us to assume regularity at $t=0$, we need that the data satisfy some non-local compatibility conditions at $t=0$ (see \cite{HR82}). These conditions are very hard to verify even for simple models and hardly appear in physical context. We have avoided these conditions in our analysis and more. We have worked here with non-smooth initial data, that is, initial velocity $\bu_0$ belong to $\bH_0^1$. Now, even the second-order spatial derivative of the velocity is not bounded in $\bL^2$, up to initial time $t=0$. Lemma \ref{est.u} tells us that $\|\bu(t)\|_2\sim O(t^{-1/2})$. So, our assumption on the initial data forces us to have only the first-order spatial derivative of the velocity bounded in $\bL^2$ at $t=0$. This proves to be a bottle-neck in our error analysis.

Our work here closely resembles the work of \cite{He04} and we present here similar results that have been obtained in \cite{He04}, but with weaker condition on $\bu_0$. In that sense, this work is devoted to the technicalities required to deal with the non-smooth initial data.

Apart from the error analysis of this two-level method, we have also discussed the error analysis of finite element Galerkin method for (\ref{nse})-(\ref{ibc}). Although, these later results are discussed in \cite{HR82}, the results here vary slightly, due to non-smooth initial data. Also, the pressure estimate has been improved for conforming finite element, in the sense that, it now reads as $\sim O(t^{-1/2})$ rather than $\sim O(t^{-1})$. For smooth initial data, it can be shown to have no singularity at $t=0$ (using similar argument presented here, e.g., using Lemma \ref{E.neg} for pressure error), rather that the estimate 
$\sim O(t^{-1/2})$ in \cite{HR82}. Similar improvement is also realized in the two-level method. Theorem $5.3$ in \cite{He04} states that the error in velocity is $\sim O(H^2)$ and the error in pressure is $\sim O(t^{-1/2}H^2)$, when the initial data is smooth. Note that whereas the velocity error remains bounded at $t=0$, it is not so for the pressure error. In our case, we have shown that
velocity and pressure errors are $\sim O(t^{-1}H^2)$ (see Lemmas \ref{eeu} and \ref{prs}), when initial data is non-smooth. By following similar procedure, for smooth initial data, we can show that the errors, in fact, are $\sim O(H^2)$. This improves the presure error presented in \cite{He04}.

Finally, we observe that it is of practical importance to discretize in time and to implement the fully discretize scheme to verify the usefulness of the present method. In fact, computational results are shown in \cite{HMR04} to support the fact that this two-level method is better that the standard Galerkin finite element method. But the analysis presented there, for a fully discrete two-level method is for smooth initial data, that is, the initial velocity $\bu_0$ is in $\bH_0^1\cap\bH^2$. We will deal with non-smooth initial data for a fully discrete two-level method somewhere else.

In the following sections, we will assume $C$ and $K$ to be generic positive constants, where $K$ generally depends on the given data, that is, $\bu_0,\f$. And for the sake of convenience, we would write $\bv(t)$ as $\bv$ when there arises no confusion.

The article is organized as follows. In section $2$, we introduce some notations and preliminaries, and present {\it a priori} estimates and regularity results for the solution of (\ref{nse})-(\ref{ibc}), when the initial data is non-smooth. Section $3$ deals with finite element Galerkin approximations, whereas Section $4$ deals with the error analysis of the Galerkin approximation. The two-level method is briefly discussed in Section $5$ and its error analysis is carried out in Section $6$.

\section{Preliminaries}
\se

For our subsequent use, we denote by bold face letters the $\R^2$-valued
function space such as
\begin{align*}
 \bH_0^1 = [H_0^1(\Omega)]^2, ~~~ \bL^2 = [L^2(\Omega)]^2.
\end{align*}
Note that $\bH^1_0$ is equipped with a norm
$$ \|\nabla\bv\|= \big(\sum_{i,j=1}^{2}(\partial_j v_i, \partial_j
 v_i)\big)^{1/2}=\big(\sum_{i=1}^{2}(\nabla v_i, \nabla v_i)\big)^{1/2}. $$
Further, we introduce some divergence free function spaces:
\begin{align*}
\bJ_1 &= \{\bphi\in\bH_0^1 : \nabla \cdot \bphi = 0\} \\
\bJ= \{\bphi \in\bL^2 :\nabla \cdot \bphi &= 0~~\mbox{in}~~
 \Omega,\bphi\cdot{\bf n}|_{\partial\Omega}=0~~\mbox {holds weakly}\}, 
\end{align*}
where ${\bf n}$ is the outward normal to the boundary $\partial \Omega$ and $\bphi
\cdot {\bf n} |_{\partial\Omega} = 0$ should be understood in the sense of trace
in $\bH^{-1/2}(\partial\Omega)$, see \cite{Tm84}.
For any Banach space $X$, let $L^p(0, T; X)$ denote the space of measurable $X$
-valued functions $\bphi$ on  $ (0,T) $ such that
$$ \int_0^T \|\bphi (t)\|^p_X~dt <\infty~~~\mbox {if}~~1 \le p < \infty, $$
and for $p=\infty$
$$ {\displaystyle{ess \sup_{0<t<T}}} \|\bphi (t)\|_X <\infty~~~\mbox {if}~~p=\infty. $$
When there arises no confusion, we simply denote these spaces by $L^p(X)$.
Through out this paper, we make the following assumptions:\\
(${\bf A1}$). For ${\bf g} \in \bL^2$, let the unique pair of solutions $\{\bv
\in\bJ_1, q \in L^2 /\R\} $ for the steady state Stokes problem
\begin{align*}
 -\Delta\bv + \nabla q = {\bf g}, \\
 \nabla \cdot\bv = 0~~~\mbox {in}~~~\Omega,~~~~\bv|_{\partial\Omega}=0,
\end{align*}
satisfy the following regularity result
$$  \| \bv \|_2 + \|q\|_{H^1/\R} \le C\|{\bf g}\|. $$
\noindent
(${\bf A2}$). The initial velocity $\bu_0$ and the external force $\f$ satisfy for
positive constant $M_0,$ and for $T$ with $0<T \leq \infty$
$$ \bu_0\in\bJ_1,~\f,\f_t,\f_{tt} \in L^{\infty} (0, T ;\bL^2)~~~\mbox{with} ~~~
\|\bu_0\|_1 \le M_0,~~{\displaystyle{\sup_{0<t<T} }}\big\{\|\f\|,\|\f_t\|,
 \|\f_{tt}\|\big\} \le M_0. $$
With $P: \bL^2-\bJ$ as orthogonal projection and $\td=P(-\Delta):\bJ_1\cap \bH^2$ as Stokes operator, we first note that $({\bf A1})$ implies
$$ \|\bv\|_2 \le C\|\td\bv\|, ~~\bv\in\bJ_1\cap\bH^2. $$
And (see \cite[(2.4)]{HR82})
\begin{equation}\label{lambda1}
\|\bv\| \le \lambda_1^{-1/2}\|\bv\|_1 \le \lambda_1^{-1}\|\bv\|_2,~~~~\bv\in\bJ_1\cap\bH^2,
\end{equation}
where $\lambda_1>0$ to be the least eigenvalue of the Stokes operator.

\noindent Before going into the details, let us introduce the weak formulation of
(\ref{nse})-(\ref{ibc}). Find a pair of functions $\{\bu(t), p(t)\},~t>0,$ such that
\begin{eqnarray}\label{wfh}
 (\bu_t, \bphi) +\nu (\nabla \bu, \nabla \bphi)+ (\bu\cdot\nabla \bu, \bphi)
 &=& ( p, \nabla \cdot \bphi) + (\f,\bphi)~~~\forall \bphi \in \bH_0^1, \\
 (\nabla \cdot \bu, \chi) &=& 0 \;\;\; \forall \chi \in L^2. \nonumber
\end{eqnarray}
Equivalently, find  $\bu(t) \in {\bf J}_1,~t>0 $ such that
\begin{equation}\label{wfj}
 (\bu_t, \bphi) +\nu (\nabla \bu, \nabla \bphi )+( \bu \cdot \nabla \bu, \bphi)
=(\f,\bphi),~\forall\bphi
 \in {\bf J}_1.
\end{equation}
For existence and uniqueness and the regularity of the solution of the problem
(\ref{wfh}) or (\ref{wfj}), we refer to \cite{HR82}. For the sake of completeness, we present below the {\it a priori} estimates and higher-order estimates, which will be used in our error analysis. These vary slightly with the results presented in \cite{HR82} due to non-smooth initial data. And hence, we sketch a proof here.
\begin{lemma}\label{est.u}
Assume that $({\bf A1})$ and $({\bf A2})$ hold and let $0<\alpha < \nu\lambda_1$. Then, for some constant $K>0$, which depends only on the given data, the weak solution pair $(\bu,p)$ of (\ref{nse})-(\ref{ibc}) satisfies the following estimates:
\begin{eqnarray}
\|\bu(t)\|^2+e^{-2\alpha t}\int_0^t e^{2\alpha s} \|\bu(s)\|_1^2 ds \le K, \label{est.u1} \\
\|\bu(t)\|_1^2+e^{-2\alpha t}\int_0^t e^{2\alpha s} \big\{\|\bu(s)\|_2^2+\|\bu_s(s)\|^2\big\} ~ds \le K, \label{est.u2} \\
\tau^*(t)\big\{\|\bu_t(t)\|^2+\|\bu(t)\|_2^2+\|p(t)\|_1^2\big\}+e^{-2\alpha t}\int_0^t \sigma(s)\|\bu_s(s)\|_1^2 ds \le K, \label{est.u3} \\
(\tau^*(t))^2\|\bu_t(t)\|_1^2+e^{-2\alpha t}\int_0^t \sigma_1(s)\big\{\|\bu_s(s)\|_2^2 +\|\bu_{ss}(s)\|^2+\|p_s(s)\|_1^2\big\}~ds \le K, \label{est.u4} \\
(\tau^*(t))^{3/2}\big\{\|\bu_{tt}(t)\|+\|\bu_t(t)\|_2^2+\|p_t(t)\|_1^2\big\} \le K, \label{est.u5}
\end{eqnarray}
where $\tau^*(t)=\min\{1,t\},~\sigma(t)=\tau^*(t)e^{2\alpha t}$ and $\sigma_1(t)= (\tau^*(t))^2 e^{2\alpha t}$.
\end{lemma}
\begin{proof}
We choose $\bphi=e^{2\alpha t}\bu= e^{\alpha t}\hbu$ in (\ref{wfj}) and use (\ref{lambda1}) to find
\begin{align}\label{est.u01}
\frac{1}{2}\frac{d}{dt}\|\hbu\|^2+\Big(\nu-\frac{\alpha}{\lambda_1}\Big)\|\hbu\|_1^2 \le
\frac{1}{\lambda_1^{1/2}}\|\hf\|\|\hbu\|_1.
\end{align}
Use kickback argument, then integrate and multiply by $e^{-2\alpha t}$ to obtain
\begin{align*}
\|\bu\|^2+\Big(\nu-\frac{\alpha}{\lambda_1}\Big)e^{-2\alpha t}\int_0^t \|\hbu(s)\|_1^2 ds \le
e^{-2\alpha t}\|\bu_0\|^2+\frac{(1-e^{-2\alpha t})}{2\alpha(\nu\lambda_1-\alpha)}
\|\f\|_{\infty}^2 =:K,
\end{align*}
where $\|\f\|_{\infty}=\|\f\|_{L^{\infty}(\bL^2)}$. \\
Now choose $\bphi=\bu$ in (\ref{wfj}) and integrate from $t$ to $t+t_0$ for some fix $t_0>0$.
\begin{align}\label{est.u02}
\|\bu(t+t_0)\|^2+\int_t^{t+t_0} \|\bu(s)\|_1^2 ds \le \|\bu(t)\|^2+\frac{\lambda_1t_0}{\nu}
\|\f\|_{\infty}^2 \le K(t_0).
\end{align}
Next, we choose $\bphi=\td\bu$ in (\ref{wfj}) to find (see \cite[(2.11)]{HR82})
\begin{align}\label{est.u03}
 \frac{d}{dt}\|\bu\|_1^2+\nu\|\td\bu\|^2 \le C\|\f\|^2+C\|\bu\|^2\|\bu\|_1^4.
\end{align}
Keeping in mind (\ref{est.u02}), we apply uniform Gronwall lemma (see \cite[Lemma 1.1, pp.
91]{Tm97}) to observe that for fixed $t_0>0$
$$ \|\bu(t+t_0)\|_1 \le K,~~t>0. $$
Multiply (\ref{est.u03}) by $e^{2\alpha t}$, integrate with respect to time, use the above estimate to conclude the first part of (\ref{est.u2}). Choose $\bphi=e^{2\alpha t}\bu_t$ for the rest of (\ref{est.u2}). \\
For next estimate, we differentiate (\ref{wfj}) with respect to time and put $\bphi=\sigma(t)
\bu_t$ to get
\begin{align*}
\frac{d}{dt}\big\{\sigma(t)\|\bu_t\|^2\big\}+2\nu\sigma(t)\|\bu_t\|_1^2 =\sigma_t(t) \|\bu_t\|^2-\sigma(t)(\bu_t\cdot\nabla \bu,\bu_t)+\sigma(t)(\f_t,\bu_t).
\end{align*}
Observe that
$$ (\bu_t\cdot\nabla \bu,\bu_t) \le \|\bu_t\|_{\bL^4(\Omega)}^2\|\bu\|_1 \le C\|\bu_t\|
\|\bu_t\|_1\|\bu\|_1. $$
And therefore, we have
\begin{align*}
\frac{d}{dt}\big\{\sigma(t)\|\bu_t\|^2\big\}+\nu\sigma(t)\|\bu_t\|_1^2 \le Ce^{2\alpha t} \|\bu_t\|^2(1+\|\bu\|_1^2)+\sigma(t)\|\f_t\|^2.
\end{align*}
Integrate with respect to time and use (\ref{est.u2}) to conclude part of (\ref{est.u3}). The rest of it can be proved by using the equations (\ref{wfj}) and (\ref{wfh}).\\
Put $\bphi=\sigma_1(t)\td\bu_t$ and $\bphi=\sigma_1(t)\bu_{tt}$, respectively, after differentiating (\ref{wfj}) and proceed in similar fashion as above to find
($\sigma_1(t)=(\tau^*(t))^2e^{2\alpha t}$)
\begin{align}\label{est.u04}
(\tau^*(t))^2\|\bu_t(t)\|_1^2+e^{-2\alpha t}\int_0^t \sigma_1(s)\big\{\|\bu_s(s)\|_2^2 +\|\bu_{ss}(s)\|^2\big\}~ds \le K.
\end{align}
Differentiate the equation (\ref{wfh}) and use (\ref{est.u04}) to obtain the pressure  estimate of (\ref{est.u4}). \\
Finally, we differentiate (\ref{wfj}) twice, with respect to time. First, put $\bphi=\sigma_2(t)\bu_{tt},~\sigma_2(t)=(\tau^*(t))^3e^{2\alpha t}$ and use (\ref{est.u4}) to obtain
$$ (\tau^*(t))^3\|\bu_{tt}(t)\|_1^2+e^{-2\alpha t}\int_0^t \sigma_2(s)\|\bu_{ss}(s)\|_1^2 ds \le K. $$
And then use the double differentiated equation and the equation (\ref{wfh}) (after double differentiation) with the above obtained estimate to conclude (\ref{est.u5}) and this completes the rest of the proof.
\end{proof}
\noindent
We present below Gronwall's Lemma, which will be used subsequently.
\begin{lemma} [Gronwall's Lemma]
 Let $g,h,y$ be three locally integrable non-negative functions  on the time
 interval $[0,\infty)$ such that for all $t\ge 0$
 $$  y(t)+G(t)\le C+\int_0^t h(s)~ds+\int_0^t g(s)y(s)~ds, $$
 where $G(t)$ is a non-negative function on $[0,\infty)$ and  $C\ge 0$ is a
 constant. Then,
 $$ y(t)+G(t)\le{\Big(}C+\int_0^t h(s)~ds{\Big)}exp{\Big(}\int_0^t g(s)~ds{\Big)}. $$
\end{lemma}

\section{Classical Galerkin Method}
\se

From now on, we denote $h$ with $0<h<1$ to be a real positive discretization
parameter tending to zero. Let  $\bH_h$ and $L_h$, $0<h<1$ be two family of
finite dimensional subspaces of $\bH_0^1 $ and $L^2/\R$, respectively,
approximating velocity vector and the pressure. Assume that the following
approximation properties are satisfied for the spaces $\bH_h$ and $L_h$: \\
${\bf (B1)}$ For each $\bw \in\bH_0^1 \cap \bH^2 $ and $ q \in
H^1/\R$ there exist approximations $i_h w \in \bH_h $ and $ j_h q \in
L_h $ such that
$$ \|\bw-i_h\bw\|+ h\|\nabla (\bw-i_h\bw)\| \le Ch^2 \| \bw\|_2,
 ~~~~\|q-j_h q\|\le Ch\|q\|_1. $$
Further, suppose that the following inverse hypothesis holds for $\bw_h\in\bH_h$:
\begin{align}\label{inv.hypo}
\|\nabla \bw_h\| \le Ch^{-1} \|\bw_h\|.
\end{align}
To define the Galerkin approximations, we set for $\bv, \bw, \bphi \in \bH_0^1$,
$$ a(\bv, \bphi) = (\nabla \bv, \nabla \bphi) $$
and
$$ b(\bv, \bw,\bphi)= \frac{1}{2} (\bv \cdot \nabla \bw , \bphi)
   - \frac{1}{2} (\bv \cdot \nabla \bphi, \bw). $$
Note that the operator $b(\cdot, \cdot, \cdot)$ preserves the antisymmetric property of
the original nonlinear term, that is,
$$ b(\bv_h, \bw_h, \bw_h) = 0 \;\;\; \forall \bv_h, \bw_h \in {\bH}_h. $$
The discrete analogue of the weak formulation (\ref{wfh}) now reads as: Find $\bu_h(t)
\in \bH_h$ and $p_h(t) \in L_h$ such that $ \bu_h(0)= \bu_{0h} $ and for $t>0$
\begin{eqnarray}\label{dwfh}
(\bu_{ht}, \bphi_h) +\nu a (\bu_h,\bphi_h) &+& b(\bu_h,\bu_h,\bphi_h) -(p_h, \nabla \cdot \bphi_h) =(\f, \bphi_h), \nonumber \\
&&(\nabla \cdot \bu_h, \chi_h) =0,
\end{eqnarray}
for $\bphi_h\in\bH_h,~\chi_h \in L_h$. Here $\bu_{0h} \in\bH_h $ is a suitable 
approximation of $\bu_0\in\bJ_1$.

\noindent In order to consider a discrete space analogous to $\bJ_1$, we
impose the discrete incompressibility condition on $\bH_h$ and call it as
$\bJ_h$. Thus, we define $\bJ_h,$ as
$$ {\bf J}_h = \{ v_h \in \bH_h : (\chi_h,\nabla\cdot v_h)=0
 ~~~\forall \chi_h \in L_h \}. $$
Note that $\bJ_h$ is not a subspace of $\bJ_1$. With $\bJ_h$ as above, we now introduce
an equivalent Galerkin formulation as: Find $\bu_h(t)\in {\bf J}_h $ such that $\bu_h(0) =
\bu_{0h} $ and for $t>0$
\begin{equation}\label{dwfj}
~~~~ (\bu_{ht},\bphi_h) +\nu a (\bu_h,\bphi_h)= -b( \bu_h, \bu_h, \bphi_h) +(\f,\bphi_h)~~\forall \bphi_h \in \bJ_h.
\end{equation}
Since $\bJ_h$ is finite dimensional, the problem (\ref{dwfj}) leads to a system of
nonlinear  ordinary differential equations. For global existence of a unique solution
of (\ref{dwfj}) (or unique solution pair of (\ref{dwfh})), we again refer to \cite{HR82}.

For continuous dependence of the discrete pressure $p_h (t) \in L_h$ on the
discrete velocity $u_h(t) \in {\bf J}_h$, we assume the following discrete
inf-sup (LBB) condition on the finite dimensional spaces $\bH_h$ and $L_h$:\\
\noindent
${\bf (B2')}$  For every $q_h \in L_h$, there exists a non-trivial function
$\bphi_h \in \bH_h$ and a positive constant $K_0,$ independent of $h,$
such that
$$ |(q_h, \nabla\cdot \bphi_h)| \ge C\|\nabla \bphi_h \|\| q_h\|_{L^2/\R}. $$
Moreover, we also assume that the following approximation property holds true
for ${\bf J}_h $. \\
\noindent
${\bf (B2)}$ For every $\bw \in {\bf J}_1 \cap \bH^2, $ there exists an
approximation $r_h \bw \in {\bf J_h}$ such that
$$ \|\bw-r_h\bw\|+h \| \nabla (\bw - r_h \bw) \| \le Ch^2 \|\bw\|_2 . $$
The $L^2$ projection $P_h:\bL^2\mapsto \bJ_h$ satisfies the following properties
(see \cite{HR82}): for $\bphi\in \bJ_h$,
\begin{equation}\label{ph1}
 \|\bphi- P_h \bphi\|+ h \|\nabla P_h \bphi\| \leq C h\|\nabla \bphi\|,
\end{equation}
and for $\bphi \in \bJ_1 \cap \bH^2,$
\begin{equation}\label{ph2}
 \|\bphi-P_h\bphi\|+h\|\nabla(\bphi-P_h \bphi)\|\le C h^2\|\td\bphi\|.
\end{equation}
We now define the discrete  operator $\Delta_h: \bH_h \mapsto \bH_h$ through the
bilinear form $a (\cdot, \cdot)$ as
\begin{eqnarray}\label{do}
 a(\bv_h, \bphi_h) = (-\Delta_h\bv_h, \bphi)~~~~\forall \bv_h, \bphi_h\in\bH_h.
\end{eqnarray}
Set the discrete analogue of the Stokes operator $\td =P(-\Delta) $ as
$\td_h = P_h(-\Delta_h) $. Note that the $\td_h$ restricted to $\bJ_h$ is invertible and we denote its inverse by $(\td_h)^{-1}$; for details, see \cite{HR82, HR90}. Following \cite{HR90}, we define {\it discrete} Sobolev norm as:
$$ \|\bv_h\|_r = \|(-\td_h)^{r/2}\bv_h\|,~~\mbox{for }\bv_h\in\bJ_h,~r\in\R. $$
Using Sobolev imbedding and Sobolev inequality, it is
easy to prove the following Lemma (similar to \cite[$(3.4)$]{HR90}).
\begin{lemma}\label{nonlin}
Suppose conditions (${\bf A1}$), (${\bf B1}$) and (${\bf  B2}$) are satisfied. Then there 
exists a positive constant $K$ such that for $\bv,\bw,\bphi\in\bH_h$, the following holds:
\begin{equation}\label{nonlin1}
 |(\bv\cdot\nabla\bw,\bphi)| \le K \left\{
\begin{array}{l}
 \|\bv\|^{1/2}\|\nabla\bv\|^{1/2}\|\nabla\bw\|^{1/2}\|\Delta_h\bw\|^{1/2}
 \|\bphi\|, \\
 \|\bv\|^{1/2}\|\Delta_h\bv\|^{1/2}\|\nabla\bw\|\|\bphi\|, \\
 \|\bv\|^{1/2}\|\nabla\bv\|^{1/2}\|\nabla\bw\|\|\bphi\|^{1/2}
 \|\nabla\bphi\|^{1/2}, \\
 \|\bv\|\|\nabla\bw\|\|\bphi\|^{1/2}\|\Delta_h\bphi\|^{1/2}, \\
 \|\bv\|\|\nabla\bw\|^{1/2}\|\Delta_h\bw\|^{1/2}\|\bphi\|^{1/2}
 \|\nabla\bphi\|^{1/2}
\end{array}\right.
\end{equation}
\end{lemma}

\noindent Examples of subspaces $\bH_h$ and $L_h$ satisfying assumptions (${\bf B1}$)
and (${\bf B2}'$) are abundant in literature, for example, see \cite{BF, BP, GR}.

\noindent We present below a couple of lemmas, one dealing with {\it a priori} and regularity estimates of $\bu_h$ and the other, with higher-order regularity results. The proof is similar to that of Lemma \ref{est.u}.
\begin{lemma}\label{est.uh}
Under the assumptions of Lemmas \ref{est.u} and \ref{nonlin}, the semi-discrete Galerkin approximation $\bu_h$ of the velocity $\bu$ satisfies, for $t>0,$
\begin{eqnarray}
\|\bu_h(t)\|^2+e^{-2\alpha t}\int_0^t e^{2\alpha s}\|\bu_h(s)\|_1^2~ds \le K, \label{uh01} \\
\|\bu_h(t)\|_1^2+e^{-2\alpha t}\int_0^t e^{2\alpha s}\big\{\|\bu_h(s)\|_2^2+\|\bu_{h,s}(s)\|^2 \big\}~ds \le K, \label{uh02} \\
(\tau^*(t))^{1/2}\big\{\|\bu_h(t)\|_2+\|p\|_{H^1/\R}\} \le K, \label{uh03}
\end{eqnarray}
where $\tau^*(t)= \min \{1,t\}$ and $K$ depends only on the given data. In particular, $K$ is independent of $h$ and $t$.
\end{lemma}
\begin{lemma}\label{est1.uh}
Under the assumptions of Lemma \ref{est.uh}, the semi-discrete Galerkin approximation $\bu_h$ of the velocity $\bu$ satisfies, for $t>0$ and for $r\in \{0,1\},~i\in \{1,2\},~r+i \le 2$,
\begin{equation}\label{1uh01}
(\tau^*(t))^{r+2i-1}\|D_t^i\bu_h(t)\|_r^2+e^{-2\alpha t}\int_0^t (\tau^*(s))^{r+2i-1}
e^{2\alpha s} \|D_s^i\bu_h(s)\|_{r+1}^2 ~ds \le K,
\end{equation}
where $D_t^i=\frac{\partial^i}{\partial t^i}$. And
\begin{eqnarray}
e^{-2\alpha t}\int_0^t (\tau^*(s))^2 e^{2\alpha s}\{\|\bu_{h,ss}(s)\|^2 +\|p_{h,s}(s)\|_{H^1/\R}^2\}~ds \le K, \label{1uh02} \\
(\tau^*(t))^{3/2}\big\{\|\bu_{ht}\|_2+\|p_{ht}(t)\|_1\big\} \le K. \label{1uh03}
\end{eqnarray}
Here $K$ depends only on the given data. In particular, $K$ is independent of $h$ and $t$.
\end{lemma}

\section{Error Analysis: Galerkin Method}
\se

In this section, we briefly present a convergence analysis for Galerkin approximation. The analysis for smooth initial data, that is, $\bu_0\in\bJ_1\cap\bH^2$ can be found in \cite{HR82}. The analysis below is for non-smooth initial data, that is, $\bu_0\in\bJ_1$.
The proofs follow similar lines as those for smooth data (see \cite{HR82}) apart from a few modifications. We will simply try to highlight these modifications in our proofs.

\begin{theorem}\label{errest}
Let $\Omega$ be a convex polygon and let the conditions (${\bf A1}$)-(${\bf A2}$)
and (${\bf B1}$)-(${\bf B2}$) be satisfied. Further, let the discrete initial velocity $\bu_{0h}\in \bJ_h$ with $\bu_{0h}=P_h\bu_0,$ where $\bu_0\in \bJ_1.$ Then,
there exists a positive constant $C$, that depends only on the given data and the
domain $\Omega$, such that for $0<T<\infty $ with $t\in (0,T]$
$$ \|(\bu-\bu_h)(t)\|+h\|\nabla(\bu-\bu_h)(t)\|\le Ke^{Kt}t^{-1/2}h^2. $$
\end{theorem}

\begin{proof}
Denoting the Galerkin approximation error as $\E=\bu-\bu_h$, we split the error in two parts.
$$ \E= (\bu-\bv_h)+(\bv_h-\bu_h) =:\bxi+\bta, $$
where $\bv_h$ satisfies the linearized equation $(5.3)$ from \cite{HR82}. We have (see \cite[Lemma 5.1]{HR82}; same proof will go through)
\begin{equation}\label{bxi.l2l2}
e^{-2\alpha t}\int_0^t e^{2\alpha s}\|\bxi(s)\|^2 ds \le Kh^4.
\end{equation}
For $L^{\infty}(\bL^2)$ estimate of $\bxi$, we now split $\bxi$ as follows:
$$ \bxi= (\bu-\bv_h)= (\bu-S_h\bu)+(S_h\bu-\bv_h) =:\bzt+\bth, $$
where $S_h$ is given by $(4.52)$, \cite{HR82}. Lemma $4.7$, \cite{HR82} tells us that
\begin{eqnarray}
\|\bzt\|+h\|\bzt\|_1 \le Ch^2\{(1+\|\bu\|_1)\|\bu\|_2+\|p\|_1\} \label{Stokes1} \\
\|\bzt_t\|+h\|\bzt_t\|_1 \le Ch^2\{(1+\|\bu\|_1)\|\bu_t\|_2+\|p_t\|_1\} \label{Stokes2}
\end{eqnarray}
In order to complete the estimate for $\bxi$, we only need to estimate $\bth$. The equation in $\bth$ reads as
\begin{equation}\label{bth}
(\bth_t,\bphi_h)+\nu a(\bth,\bphi_h)= -(\bzt_t,\bphi_h),~~\bphi_h\in\bJ_h.
\end{equation}
For $\sigma(t)=\tau^*(t)e^{2\alpha t}$, we put $\bphi_h=\sigma(t)\bth$ in (\ref{bth}) to find
\begin{align*}
\frac{d}{dt}\big\{\sigma(t)\|\bth\|^2\big\}-\sigma_t(t)\|\bth\|^2+2\nu\sigma(t)\|\bth\|_1^2
= -2\sigma(t)(\bzt_t,\bth).
\end{align*}
Using Cauchy-Schwarz and Young's inequality, we obtain
\begin{align*}
\frac{d}{dt}\big\{\sigma(t)\|\bth\|^2\big\}+2\nu\sigma(t)\|\bth\|_1^2 & \le Ce^{2\alpha t} \|\bth\|^2+\sigma_1(t)\|\bzt_t\|^2 \\
& \le Ce^{2\alpha t}\big\{\|\bxi\|^2+\|\bzt\|^2\big\}+\sigma_1(t)\|\bzt_t\|^2.
\end{align*}
Integrate with respect to time and use (\ref{bxi.l2l2}), (\ref{Stokes1}) and (\ref{Stokes2}).
\begin{align*}
\sigma(t)\|\bth(t)\|^2+\nu\int_0^t \sigma(s)\|\bth(s)\|_1^2 ds &\le Ke^{2\alpha t}h^4
+Ch^4\int_0^t e^{2\alpha s}\big\{\|\bu(s)\|_2^2+\|p(s)\|_1^2\big\}~ds \\
& +Ch^4\int_0^t \sigma_1(s)\big\{\|\bu_s(s)\|_2^2 +\|p_s(s)\|_1^2\big\}
\end{align*}
Use (\ref{est.u2}) and (\ref{est.u4}) and then multiply by $e^{-2\alpha t}$ to get
\begin{align*}
\tau^*(t)\|\bth(t)\|^2+e^{-2\alpha t}\int_0^t \sigma(s)\|\bth(s)\|_1^2 ds \le Kh^4.
\end{align*}
Now an use of triangle inequality along with inverse hypothesis (\ref{inv.hypo}) results in
\begin{align}\label{bxi.l2}
\|\bxi(t)\|+h\|\bxi(t)\|_1 \le Kt^{-1/2}h^2.
\end{align}
It now remains to estimate $\bta$. The equation in $\bta$ is
\begin{equation}\label{bta1}
(\bta_t,\bphi_h)+\nu a(\bta,\bphi_h)= \Lambda_{1,h}(\bphi_h),~~\bphi_h\in\bJ_h,
\end{equation}
where
\begin{align}\label{lamb.1h}
\Lambda_{1,h}(\bphi_h)= b(\bu_h,\bu_h,\bphi_h)-b(\bu,\bu,\bphi_h)
= -b(\bu_h,\E,\bphi_h)-b(\E,\bu,\bphi_h).
\end{align}
Due to non-smooth initial data, we need an intermediate estimate, before we proceed for $L^{\infty}(\bL^2)$ estimate of $\bta$. First we need $L^2(\bL^2)$ estimate of $\bta$. For that, choose $\bphi_h=e^{2\alpha t}(-\td_h)^{-1}\bta$ in (\ref{bta1}) to obtain
\begin{align}\label{bta1.neg}
\frac{1}{2}\frac{d}{dt}\|\hta\|_{-1}^2-\alpha\|\hta\|_{-1}^2+\nu\|\hta\|^2= e^{2\alpha t}
\Lambda_{1,h}((-\td_h)^{-1}\bta).
\end{align}
Here, $\hta= e^{\alpha t}\bta$.
Using (\ref{lamb.1h}) and the definition of $b(\cdot,\cdot,\cdot)$, we observe that
\begin{align}\label{nonlin01}
\Lambda_{1,h}((-\td_h)^{-1}\bta)=& -b(\bu_h,\bxi+\bta,(-\td_h)^{-1}\bta)
-b(\bxi+\bta,\bu,(-\td_h)^{-1}\bta) \nonumber \\
=& -\frac{1}{2}((\bu_h\cdot\nabla)(\bxi+\bta),(-\td_h)^{-1}\bta)-\frac{1}{2} ((\bu_h\cdot\nabla)(-\td_h)^{-1}\bta,\bxi+\bta) \nonumber \\
& -\frac{1}{2}(((\bxi+\bta)\cdot\nabla)\bu,(-\td_h)^{-1}\bta)-\frac{1}{2} (((\bxi+\bta)\cdot\nabla)(-\td_h)^{-1}\bta,\bu).
\end{align}
Use Lemma \ref{nonlin} to arrive at the following:
\begin{eqnarray}
-\frac{1}{2} ((\bu_h\cdot\nabla)(-\td_h)^{-1}\bta,\bxi+\bta) &\le & C\|\bu_h\|^{1/2}
\|\bu_h\|_1^{1/2}\|\bta\|_{-1}^{1/2}\|\bta\|^{1/2}(\|\bxi\|+\|\bta\|) \label{nonlin01a} \\
-\frac{1}{2}(((\bxi+\bta)\cdot\nabla)\bu,(-\td_h)^{-1}\bta) &\le & C(\|\bxi\|+\|\bta\|)
\|\bu\|_1\|\bta\|_{-2}^{1/2}\|\bta\|^{1/2} \label{nonlin01b} \\
-\frac{1}{2} (((\bxi+\bta)\cdot\nabla)(-\td_h)^{-1}\bta,\bu) &\le & C(\|\bxi\|+\|\bta\|)
\|\bta\|_{-1}^{1/2}\|\bta\|^{1/2}\|\bu\|^{1/2}\|\bu\|_1^{1/2}. \label{nonlin01c}
\end{eqnarray}
For the first term on the right-hand side of (\ref{nonlin01}), we have \\
({\it with the notations $D_i=\frac{\partial}{\partial x_i}$ and $\bv=(v_1,v_2)$})
\begin{align}\label{nonlin02}
((\bu_h\cdot\nabla)(\bxi+&\bta), (-\td_h)^{-1}\bta) =\sum_{i,j=1}^2 \int_{\Omega} u_{h,i}
D_i\{(\xi+\eta)_j\}((-\td_h)^{-1}\bta)_j~d{\bf x} \nonumber \\
&= -\sum_{i,j=1}^2 \int_{\Omega} D_i u_{h,i}(\xi+\eta)_j((-\td_h)^{-1}\bta)_j~d{\bf x} -\sum_{i,j=1}^2 \int_{\Omega} u_{h,i}D_i \{((-\td_h)^{-1}\bta)_j\}(\xi+\eta)_j \nonumber \\
&= -((\nabla\cdot\bu_h)(\bxi+\bta),(-\td_h)^{-1}\bta)-((\bu_h\cdot\nabla)(-\td_h)^{-1}\bta, (\bxi+\bta)) \nonumber \\
&\le C\|\bu\|_1(\|\bxi\|+\|\bta\|)\|\bta\|_{-2}^{1/2}\|\bta\|^{1/2}+C\|\bu_h\|^{1/2}
\|\bu_h\|_1^{1/2}\|\bta\|_{-1}^{1/2}\|\bta\|^{1/2}(\|\bxi\|+\|\bta\|)
\end{align}
in view of Lemma \ref{nonlin}. Incorporate (\ref{nonlin01a})-(\ref{nonlin02}) in (\ref{nonlin01}) and use Lemma \ref{est.u} to find
\begin{align}\label{nonlin03}
\Lambda_{1,h}((-\td_h)^{-1}\bta) &\le C\|\bxi\|\|\bta\|_{-1}^{1/2}\|\bta\|^{1/2} \big\{\|\bu_h\|_1+\|\bu\|_1\big\}+C\|\bta\|_{-1}^{1/2}\|\bta\|^{3/2}\|\bu\|_1 \nonumber \\
&\le \varepsilon\|\bta\|^2+K(\|\bta\|_{-1}^2+\|\bxi\|^2).
\end{align}
Hence, with appropriate $\varepsilon$, we obtain from (\ref{bta1.neg})
\begin{align*}
\frac{d}{dt}\|\hta\|_{-1}^2+\nu\|\hta\|^2= K(\|\hta\|_{-1}^2+\|\hxi\|^2).
\end{align*}
Integrate and use (\ref{bxi.l2l2}). Apply Gronwall's lemma to coclude that
\begin{equation}\label{bta.neg}
\|\bta(t)\|_{-1}^2+e^{-2\alpha t}\int_0^t e^{2\alpha s}\|\bta(s)\|^2 ds \le Ke^{Kt}h^4.
\end{equation}
Now choose $\bphi_h=\sigma(t)\bta$ in (\ref{bta1}) to obtain
\begin{align*}
\frac{d}{dt}\big\{\sigma(t)\|\bta\|^2\big\}+2\nu\sigma(t)\|\bta\|_1^2 = \sigma_t(t)\|\bta\|^2
+2\sigma(t)\Lambda_{1,h}(\bta).
\end{align*}
As in (\ref{nonlin01}), we have
\begin{align}\label{nonlin04}
\Lambda_{1,h}(\bta) &= -b(\bu_h,\bxi,\bta)-b(\bxi+\bta,\bu,\bta) \nonumber \\
=& -\frac{1}{2}((\bu_h\cdot\nabla)\bxi,\bta)-\frac{1}{2}((\bu_h\cdot\nabla)\bta,\bxi)
-\frac{1}{2}(((\bxi+\bta)\cdot\nabla)\bu,\bta)-\frac{1}{2}(((\bxi+\bta)\cdot\nabla)\bta,\bu).
\end{align}
Using Lemma \ref{nonlin}, we find
\begin{eqnarray}
&-&\frac{1}{2}(((\bxi+\bta)\cdot\nabla)\bu,\bta)-\frac{1}{2}(((\bxi+\bta)\cdot\nabla)\bta,\bu) \nonumber \\
&=& -\frac{1}{2}((\bxi\cdot\nabla)\bu,\bta)-\frac{1}{2}((\bxi\cdot\nabla)\bta,\bu)
-\frac{1}{2}((\bta\cdot\nabla)\bu,\bta)-\frac{1}{2}((\bta\cdot\nabla)\bta,\bu) \nonumber \\
&\le & C\|\bxi\|\{\|\bu\|_1^{1/2}\|\bu\|_2^{1/2}\|\bta\|^{1/2}\|\bta\|_1^{1/2}+\|\bta\|_1
\|\bu\|^{1/2}\|\bu\|_2^{1/2}\} \nonumber \\
&&+C\|\bta\|\|\bta\|_1\|\bu\|_1+C\|\bta\|^{1/2}\|\bta\|_1^{3/2}\|\bu\|^{1/2}\|\bu\|_1^{1/2}. \label{nonlin04a}\\
&-& \frac{1}{2}((\bu_h\cdot\nabla)\bta,\bxi) \le C\|\bu_h\|^{1/2}\|\bu_h\|_2^{1/2}\|\bta\|_1
\|\bxi\|. \label{nonlin04b}
\end{eqnarray}
And as for the first term on the right-hand side of (\ref{nonlin04}), we observe, as in (\ref{nonlin02}),
\begin{align}\label{nonlin05}
-\frac{1}{2}((\bu_h\cdot\nabla)\bxi,\bta)& = -\frac{1}{2}((\nabla\cdot\bu_h)\bxi,\bta)
-\frac{1}{2}((\bu_h\cdot\nabla)\bta,\bxi) \nonumber \\
&\le C\|\bu_h\|_1^{1/2}\|\bu_h\|_2^{1/2}\|\bxi\|\|\bta\|^{1/2}\|\bta\|_1^{1/2}+
C\|\bu_h\|^{1/2}\|\bu_h\|_2^{1/2}\|\bta\|_1\|\bxi\|.
\end{align}
Incorporating (\ref{nonlin04a})-(\ref{nonlin05}) in (\ref{nonlin04}) and using Lemmas \ref{est.u} and \ref{est.uh}, we obtain
\begin{align}
\Lambda_{1,h}(\bta) &\le K\|\bxi\|\|\bta\|_1(\tau^*(t))^{-1/4}+K(\|\bta\|\|\bta\|_1+
\|\bta\|^{1/2}\|\bta\|_1^{3/2}) \nonumber \\
&\le K\|\bxi\|\|\bta\|_1(\tau^*(t))^{-1/4}+\varepsilon\|\bta\|_1^2+K\|\bta\|^2,
\end{align}
for some $\varepsilon>0$. Therefore,
\begin{align*}
2\sigma(t)\Lambda_{1,h}(\bta) \le 4\varepsilon\sigma(t)\|\bta\|_1^2+K\sigma(t)\|\bta\|^2
+Ke^{2\alpha t}\|\bxi\|^2.
\end{align*}
With appropriate $\varepsilon$, we now obtain
\begin{align*}
\frac{d}{dt}\big\{\sigma(t)\|\bta\|^2\big\}+\nu\sigma(t)\|\bta\|_1^2 \le
K\big\{\|\hxi\|^2+\|\hta\|^2\big\}.
\end{align*}
Integrate and use (\ref{bxi.l2l2}) and (\ref{bta.neg}) to find that
\begin{equation}\label{bta.l2a}
\tau^*(t)\|\bta(t)\|^2+e^{-2\alpha t}\int_0^t \sigma(s)\|\bta(s)\|_1^2 ds \le Ke^{Kt}h^4.
\end{equation}
With inverse hypothesis (\ref{inv.hypo}), we conclude that
\begin{equation}\label{bta.l2}
\|\bta(t)\|+h\|\bta(t)\|_1 \le Ke^{Kt}t^{-1/2}h^2.
\end{equation}
This along with (\ref{bxi.l2}) completes the proof.
\end{proof}

\begin{theorem}\label{errest.p}
Under the assumptions of Theorem \ref{errest} and with an additional assumption of $({\bf B2}')$, we have
$$ \|(p-p_h)(t)\| \le Ke^{Kt}t^{-1/2}h. $$
\end{theorem}
\begin{proof}
Following Lemma $6.1$, \cite{HR82}, and using $({\bf B1})$, (\ref{est.u3}) and $({\bf B2}')$ along with the observation that
$$ \Lambda_{1,h}(\bphi_h) \le K\|\E\|_1\|\bphi_h\|_1 $$
we obtain
\begin{align}\label{errest.p01}
\|(p-p_h)(t)\|_{L^2/\R} \le Kt^{-1/2}h+K\|\E(t)\|_1+C\|\E_t(t)\|_{-1,h},
\end{align}
where
\begin{align*}
\|\E_t\|_{-1,h} =\sup_{0\neq \bphi_h\in\bH_h}\frac{<\E_t,\bphi_h>}{\|\bphi_h\|_1}.
\end{align*}
Taking supremum over a bigger set, we find that
\begin{align}\label{neg.h}
\|\E_t\|_{-1,h} \le \|\E_t\|_{-1} =\sup_{0\neq\bphi\in\bH_0^1} \frac{<\E_t,\bphi>}{\|\bphi\|_1}.
\end{align}
This is possible due to conforming finite elements. We now have from (\ref{errest.p01}), using Theorem \ref{errest},
$$ \|(p-p_h)(t)\|_{L^2/\R} \le Ke^{Kt}t^{-1/2}h+C\|\E_t(t)\|_{-1}. $$
We complete the proof by proving the following Lemma.
\end{proof}

\begin{lemma}\label{E.neg}
The error $\E=\bu-\bu_h$ due to Galerkin approximation satisfies, for $t>0$
$$ \|\E_t(t)\|_{-1} \le Ke^{Kt}t^{-1/2}h. $$
\end{lemma}

\begin{proof}
The equation in $\E$ reads as
\begin{align}\label{Err}
(\E_t,\bphi_h)+\nu a(\E,\bphi_h)= \Lambda_{1,h}(\bphi_h)+(p,\nabla\cdot\bphi_h),
~~\bphi_h\in\bJ_h.
\end{align}
Now, for any $\bphi\in\bH_0^1$, we use (\ref{Err}) to find
\begin{align}\label{E.neg01}
(\E_t,\bphi) &= (\E_t,P_h\bphi)+(\E_t,\bphi-P_h\bphi) \nonumber \\
&= -\nu a(\E,P_h\bphi)+\Lambda_{1,h}(P_h\bphi)+(p,\nabla\cdot P_h\bphi)+(\E_t,\bphi-P_h\bphi).
\end{align}
Using discrete incompressibility condition, $H^1$-stability of $P_h$ and $({\bf B1})$, we obtain
\begin{equation}\left\{\begin{array}{rcl}\label{E.neg01a}
-\nu a(\E,P_h\bphi) &\le & C\|\E\|_1\|\bphi\|_1 \\
(p,\nabla\cdot P_h\bphi) &=& (p-j_hp,\nabla\cdot P_h\bphi) \le Ch\|p\|_1\|\bphi\|_1 \\
\Lambda_{1,h}(P_h\bphi) &\le & C\|\E\|_1\|\bphi\|_1(\|\bu\|_1+\|\bu_h\|_1) \\
(\E_t,\bphi-P_h\bphi) &=& (\bu_t,\bphi-P_h\bphi) \le Ch\|\bu_t\|\|\bphi\|_1.
\end{array}\right.
\end{equation}
From (\ref{E.neg01}), we now have, using the definition of negative norm
\begin{align}\label{E.neg02}
\|\E_t\|_{-1} \le C\|\E\|_1+Ch\|p\|_1+K\|\E\|_1+Ch\|\bu_t\|.
\end{align}
Use Lemma \ref{est.u} and Theorem \ref{errest} to complete the rest of the proof.
\end{proof}
\begin{remark}
Instead of using (\ref{neg.h}), if we have used a crude estimate, like $\|\E_t\|_{-1,h} \le C\|\E_t\|$ (as in \cite[$(6.2)$]{HR82}), we have obtained
$$ \|(p-p_h)(t)\| \le Ke^{Kt}t^{-1}h. $$
In that sense, we have improved the pressure error estimate for conforming finite elements.
\end{remark}

\begin{remark}
In the process, we observe the following estimate, combining (\ref{bxi.l2l2}) and (\ref{bta.neg}):
\begin{equation}\label{err.l2l2}
e^{-2\alpha t}\int_0^t e^{2\alpha s}\|\E(s)\|^2 ds \le Ke^{Kt}h^4.
\end{equation}
\end{remark}
\noindent
Below we present an estimate for $\E_t$, which will be needed in our error analysis for two-level method.
\begin{theorem}\label{errest.t}
Under the assumptions of Theorem \ref{errest}, we have
\begin{equation}\label{errest.t1}
e^{-2\alpha t}\int_0^t \sigma_1(s)\|\E_s(s)\|^2 ds \le Ke^{Kt}h^4.
\end{equation}
\end{theorem}
\begin{proof}
The proof is similar to that of Theorem \ref{errest} and so, we will only sketch a proof here.
First we split the error:
$$ \E_t=\bxi_t+\bta_t. $$
The equation in $\bxi_t$ is given by
\begin{align}\label{bxit}
(\bxi_{tt},\bphi_h)+\nu a(\bxi_t,\bphi_h)= (p_t,\nabla\cdot\bphi_h).
\end{align}
With $P_h:\bL^2(\Omega)\to \bJ_h$ as $L^2$-projection, we choose $\bphi_h= \sigma_1(t) P_h\bxi_t=\sigma_1(t)(\bxi_t-(\bu-P_h\bu)_t)$ in (\ref{bxit}) to find (recall $\sigma_1(t)= (\tau^*(t))^2 e^{2\alpha t},~\tau^*(t)=\min \{1,t\}$)
\begin{align*}
\frac{1}{2}\frac{d}{dt}\big\{\sigma_1(t)\|\bxi_t\|^2\}+\nu\sigma_1(t)\|\bxi_t\|_1^2=
\frac{1}{2}\sigma_{1,t}(t)\|\bxi_t\|^2+\sigma_1(t)(\bxi_{tt},(\bu-P_h\bu)_t) \\
+\nu\sigma_1(t) a(\bxi_t,(\bu-P_h\bu)_t)+\sigma_1(t)(p_t-j_hp_t,\nabla\cdot P_h\bxi_t).
\end{align*}
Use projection properties $({\bf B1})$ and Cauchy-Schwarz inequality to obtain
\begin{align}\label{errest.t01}
\frac{1}{2}\frac{d}{dt}\big\{\sigma_1(t)\|\bxi_t\|^2\}+\nu\sigma_1(t)\|\bxi_t\|_1^2 \le
C\sigma(t)\|\bxi_t\|^2+\frac{\sigma_1(t)}{2}\frac{d}{dt}\|(\bu-P_h\bu)_t\|^2 \nonumber \\
+Ch.\sigma_1(t)\|\bxi_t\|_1\big\{\|\bu_t\|_2+\|p_t\|_1\big\}
\end{align}
(recall that $\sigma(t)=\tau^*(t) e^{2\alpha t},~\tau^*(t)=\min \{1,t\}$) \\
Use Young's inequality and kickback argument. Integrate the resulting inequality to find
\begin{align*}
\sigma_1(t)\|\bxi_t\|^2+\nu\int_0^t \sigma_1(s)\|\bxi_s(s)\|_1^2 ds \le
C\int_0^t \sigma(s)\|\bxi_s(s)\|^2 ds+Ch.\sigma_1(t)\|\bu_t\|_1^2 \\
+Ch\int_0^t \sigma(s) \|\bu_s(s)\|_1^2 ds+Ch^2\int_0^t \sigma_1(s) \big\{\|\bu_t\|_2^2
+\|p_t\|_1^2\big\}~ds.
\end{align*}
Use Lemma \ref{est.u} to observe that
\begin{equation}\label{errest.t02}
(\tau^*(t))^2\|\bxi_t\|^2+\nu e^{-2\alpha t}\int_0^t \sigma_1(s)\|\bxi_s(s)\|_1^2 ds \le
Ce^{-2\alpha t}\int_0^t \sigma(s)\|\bxi_s(s)\|^2 ds+Kh^2.
\end{equation}
To estimate the first term on the right hand-side of (\ref{errest.t02}), we first recall the equation in $\bxi$:
\begin{align*}
(\bxi_t,\bphi_h)+\nu a(\bxi,\bphi_h)=(p,\nabla\cdot\bphi_h),~~\forall\bphi_h\in \bJ_h.
\end{align*}
Choose $\bphi_h=\sigma(t)P_h\bxi_t$ above.
\begin{align*}
\sigma(t)\|\bxi_t\|^2+\frac{\nu}{2}\frac{d}{dt}\big\{\sigma(t)\|\bxi\|_1^2\} \le
Ce^{2\alpha t}\|\bxi\|_1^2+Ch.\sigma(t)\big\{\|\bxi_t\|\|\bu_t\|_1
+\|\bxi\|_1\|\bu_t\|_2+\|p\|_1\|\bxi_t\|_1\big\}.
\end{align*}
Use kickback argument and then integrate with respect to time. Use Lemma \ref{est.u}
to obtain
\begin{align*}
\int_0^t \sigma(s)\|\bxi_s(s)\|^2 ds+\sigma(t)\|\bxi(t)\|_1^2 \le C\int_0^t e^{2\alpha s}
\|\bxi(s)\|_1^2 ds+Ke^{2\alpha t}h^2+\frac{\nu}{2C}\int_0^t \sigma_1(s)\|\bxi_s(s)\|_1^2 ds.
\end{align*}
Use ($5.7$) from \cite{HR82}, that is,
\begin{equation}\label{errest.t03}
e^{-2\alpha t}\int_0^t e^{2\alpha s}\|\bxi(s)\|_1^2 ds \le Kh^2
\end{equation}
to find
\begin{equation}\label{errest.t04}
e^{-2\alpha t}\int_0^t \sigma(s)\|\bxi_s(s)\|^2 ds+\tau^*(t)\|\bxi\|_1^2 \le Kh^2
+\frac{\nu}{2C}e^{-2\alpha t}\int_0^t \sigma_1(s)\|\bxi_s(s)\|_1^2 ds.
\end{equation}
Using (\ref{errest.t04}) in (\ref{errest.t02}) leads us to
\begin{equation}\label{errest.t05}
(\tau^*(t))^2\|\bxi_t\|^2+ e^{-2\alpha t}\int_0^t \sigma_1(s)\|\bxi_s(s)\|_1^2 ds \le Kh^2.
\end{equation}
Next, we use parabolic duality argument to establish (as in Lemma $5.1$ from \cite{HR82}, except that the right hand-side of (5.8), \cite{HR82} should now read $\sigma_1(t)\bxi_t$)
\begin{equation}\label{errest.t06}
e^{-2\alpha t}\int_0^t \sigma_1(s)\|\bxi_s(s)\|^2 ds \le Kh^4.
\end{equation}
An estimate of $\bta_t$ would now complete the proof. By definition, we can easily deduce the equation satisfied by $\bta_t$.
\begin{align}\label{btat}
(\bta_{tt},\bphi_h)+\nu a(\bta_t,\bphi_h)= \Lambda_{1,h,t}(\bphi_h),
\end{align}
where $\bphi_{1,h,t}$ is given by
\begin{align}\label{errest.t07}
\Lambda_{1,h,t}(\bphi_h)= -b(\bu_{h,t},\E,\bphi_h)-b(\bu_h,\E_t,\bphi_h) -b(\E_t,\bu,\bphi_h)-b(\E,\bu_t,\bphi_h).
\end{align}
Choose $\bphi_h=\sigma_1(t)(-\td_h)^{-1}\bta_t$ to find
\begin{align}\label{errest.t08}
\frac{1}{2}\frac{d}{dt}\big\{\sigma_1(t)\|\bta_t\|_{-1}^2\}+\nu\sigma_1(t)\|\bta_t\|^2 \le
C\sigma(t)\|\bta_t\|_{-1}^2+\sigma_1(t)\Lambda_{1,h,t}((-\td_h)^{-1}\bta_t).
\end{align}
Using similar proof technique of (\ref{nonlin03}), we estimate the non-linear term $\Lambda_{1,h,t}$ as follows:
\begin{align}\label{errest.t09}
& \sigma_1(t)b(\bu_{h,t},\E,(-\td_h)^{-1}\bta_t)+\sigma_1(t)b(\E,\bu_t,(-\td_h)^{-1}\bta_t) \nonumber \\
\le & C\sigma_1(t)\|\E\|\|\bta_t\|_{-1}^{1/2}\|\bta_t\|^{1/2}(\|\bu_{h,t}\|_1
+\|\bu_t\|_1) \nonumber \\
\le & \varepsilon \sigma_1(t)\|\bta_t\|^2+C\sigma(t)\|\bta_t\|_{-1}^2+Ke^{2\alpha t}\|\E\|^2
\end{align}
and
\begin{align}\label{errest.t10}
& \sigma_1(t)b(\bu_h,\E_t,(-\td_h)^{-1}\bta_t)+\sigma_1(t)b(\E_t,\bu,(-\td_h)^{-1}\bta_t)
 \nonumber \\
\le & C\sigma_1(t)(\|\bxi_t\|+\|\bta_t\|)\|\bta_t\|_{-1}^{1/2}\|\bta_t\|^{1/2}(\|\bu_h\|_1
+\|\bu\|_1) \nonumber \\
\le & K\sigma_1(t)\|\bxi_t\|\|\bta_t\|_{-1}^{1/2}\|\bta_t\|^{1/2}+K\sigma_1(t)
\|\bta_t\|_{-1}^{1/2}\|\bta_t\|^{3/2} \nonumber \\
\le & \varepsilon \sigma_1(t)\|\bta_t\|^2+K\sigma(t)\|\bta_t\|_{-1}^2+K\sigma_1(t)
\|\bxi_t\|^2.
\end{align}
Incorporate (\ref{errest.t09}) and (\ref{errest.t10}) in (\ref{errest.t08}). Use Young's inequality and appropriate $\varepsilon>0$ and finally, integrate with respect to time to find
\begin{align*}
\sigma_1(t)\|\bta_t(t)\|_{-1}^2+\nu\int_0^t \sigma_1(s)\|\bta_s(s)\|^2 ds \le
K\int_0^t \sigma(s)\|\bta_s(s)\|_{-1}^2 ds+K\int_0^t e^{2\alpha s}\|\E(s)\|^2 ds \\
+K\int_0^t \sigma_1(s)\|\bxi_s(s)\|^2 ds.
\end{align*}
Use (\ref{err.l2l2}) and (\ref{errest.t06}) to obtain
\begin{align}\label{errest.t11}
\sigma_1(t)\|\bta_t(t)\|_{-1}^2+\nu\int_0^t \sigma_1(s)\|\bta_s(s)\|^2 ds \le K(t)e^{Kt}h^4
+K\int_0^t \sigma(s)\|\bta_s(s)\|_{-1}^2 ds.
\end{align}
To estimate the last term of (\ref{errest.t11}), we recall the equation in $\bta$ (\ref{bta1}):
$$ (\bta_t,\bphi_h)+\nu a(\bta,\bphi_h)= \Lambda_{1,h}(\bphi_h). $$
Choose $\bphi_h=\sigma(t)(-\td_h)^{-1}\bta_t$ to find
\begin{align}\label{errest.t12}
\sigma(t)\|\bta_t\|_{-1}^2+\frac{\nu}{2}\frac{d}{dt}\{\sigma(t)\|\bta\|^2\} \le 
\frac{\nu}{2}\sigma_t(t)\|\bta\|^2+\sigma(t)\Lambda_{1,h}((-\td_h)^{-1}\bta_t).
\end{align}
Similar to (\ref{nonlin03}), we estimate the non-linear term as follows:
\begin{align*}
\sigma(t)\Lambda_{1,h}((-\td_h)^{-1}\bta_t) &= -\sigma(t) b(\E,\bu_h,(-\td_h)^{-1}\bta_t)
-\sigma(t) b(\bu, \E,(-\td_h)^{-1}\bta_t) \\
&\le C\sigma(t)\|\E\|\|\bta_t\|_{-1}\{\|\bu_h\|_2+\|\bu\|_2\}
\end{align*}
Incorporate this in (\ref{errest.t12}), use kickback argument and then integrate with respect to time. Next, use Lemmas \ref{est.u} and \ref{est.uh}.
\begin{align}\label{err.et14}
\int_0^t \sigma(s)\|\bta_s(s)\|_{-1}^2 ds+\sigma(t)\|\bta(t)\|^2 \le 
C\int_0^t e^{2\alpha s}\|\bta(s)\|^2 ds+K\int_0^t e^{2\alpha s}\|\E(s)\|^2 ds.
\end{align}
Apply (\ref{bta.neg}) and (\ref{err.l2l2}) in (\ref{err.et14}). Use the resulting estimate in (\ref{errest.t11}) to obtain
\begin{align}\label{errest.t13}
\sigma_1(t)\|\bta_t(t)\|_{-1}^2+\nu\int_0^t \sigma_1(s)\|\bta_s(s)\|^2 ds \le K(t)e^{Kt}h^4.
\end{align}
This along with (\ref{errest.t06}) now completes the rest of the proof.
\end{proof}

\noindent The pressure error estimate for the two-level method requires another estimate.

\begin{lemma}\label{low.est}
Under the assumptions of Theorem \ref{errest}, we have
\begin{equation}\label{low.est1}
\int_0^t \sigma_1(s)\|\E_s(s)\|_1^2 ds \le Ke^{Kt}h^2.
\end{equation}
\end{lemma}
\begin{proof}
Keeping in mind (\ref{errest.t05}), we only need the estimate in $\bta$.
Choosing $\bphi_h=\sigma_1(t)\bta_t$ in (\ref{errest.t07}) and estimating non-linear as
\begin{align*}
\Lambda_{1,h,t}(\bta_t) \le \varepsilon\|\bta_t\|_1^2+C\|\E\|_1^2(\|\bu_{h,t}\|_1 +\|\bu_t\|_1)+C\|\bxi_t\|_1^2(\|\bu_h\|_1^2+\|\bu\|_1^2)+K\|\bta_t\|^2,
\end{align*}
we find
\begin{align*}
\sigma_1(t)\|\bta_t(t)\|^2+\nu\int_0^t \sigma_1(s)\|\bta_s(s)\|_1^2 ds \le K\int_0^t \sigma(s)\|\bta_s(s)\|^2 ds+K\int_0^t \sigma_1(s)\|\bxi_s(s)\|_1^2 ds \nonumber \\
+Ke^{Kt}h^2\int_0^t \sigma(s)\big\{\|\bu_{h,s}(s)\|_1^2+\|\bu_s(s)\|_1^2\} ds
\end{align*}
Using (\ref{errest.t05}) and Lemmas \ref{est.u} and \ref{est1.uh}, we have
\begin{align}\label{rem.01}
\sigma_1(t)\|\bta_t(t)\|^2+\nu\int_0^t \sigma_1(s)\|\bta_s(s)\|_1^2 ds \le K\int_0^t
\sigma(s)\|\bta_s(s)\|^2 ds+Ke^{Kt}h^2.
\end{align}
Now, for the first term on the right-hand side, we take $\bphi_h=\sigma(t)\bta_t$ in the $\bta$ equation, to obtain, as in (\ref{errest.t12}):
\begin{align*}
\sigma(t)\|\bta_t\|^2+\frac{\nu}{2}\frac{d}{dt}\{\sigma(t)\|\bta\|_1^2\} \le 
\frac{\nu}{2}\sigma_t(t)\|\bta\|_1^2+\sigma(t)\Lambda_{1,h}(\bta_t).
\end{align*}
We estimate the non-linear term as:
\begin{align*}
\Lambda_{1,h}(\bta_t) \le K\|\bta_t\|\|\E\|_1(\|\bu_h\|_2^{1/2}+\|\bu\|_2^{1/2})
\end{align*}
to find
\begin{align}\label{rem.02}
\sigma(t)\|\bta(t)\|_1^2+\int_0^t \sigma(s)\|\bta_s(s)\|^2 ds \le Ke^{Kt}h^2+
C\int_0^t e^{2\alpha s}\|\bta(s)\|_1^2 ds.
\end{align}
To estimate the last term, choose $\bphi_h=e^{2\alpha t}\bta$ in the $\bta$ equation and estimate the non-linear term as:
\begin{align*}
\Lambda_{1,h}(\bta) \le \varepsilon\|\bta\|_1+K(\|\bxi\|_1^2+\|\bta\|^2)
\end{align*}
to obtain
\begin{align*}
\frac{d}{dt}\|\hta\|^2+\nu\|\hta\|_1^2 \le K(\|\hxi\|_1^2+\|\hta\|^2)
\end{align*}
Integrate with respect to time and the use $(5.7)$, \cite{HR82}, that is
\begin{equation}\label{rem.03}
\int_0^t e^{2\alpha s}\|\bxi(s)\|_1^2 ds \le Ke^{2\alpha t}h^2,
\end{equation}
to observe
$$ \|\hta(t)\|^2+\int_0^t e^{2\alpha s}\|\bta(s)\|_1^2 ds \le Ke^{2\alpha t}h^2
+\int_0^t e^{2\alpha s}\|\bta(S)\|^2 ds. $$
Use Gronwall's lemma.
\begin{equation}\label{rem.04}
\|\hta(t)\|^2+\int_0^t e^{2\alpha s}\|\bta(s)\|_1^2 ds \le Ke^{K t}h^2.
\end{equation}
Combining (\ref{rem.03}) and (\ref{rem.04}), we conclude
\begin{align}\label{rem.05}
\int_0^t e^{2\alpha s}\|\E(s)\|_1^2 ds \le Ke^{Kt}h^2.
\end{align}
Now from (\ref{rem.02}), we have
\begin{align}\label{rem.06}
\sigma(t)\|\bta(t)\|_1^2+\int_0^t \sigma(s)\|\bta_s(s)\|^2 ds \le Ke^{Kt}h^2.
\end{align}
And from (\ref{rem.01}), we have
\begin{align}\label{rem.07}
\sigma_1(t)\|\bta_t(t)\|^2+\nu\int_0^t \sigma_1(s)\|\bta_s(s)\|_1^2 ds \le Ke^{Kt}h^2.
\end{align}
Combining with (\ref{errest.t05}), we complete the rest of the proof.
\end{proof}

\section{Two-Level FE Galerkin Method}
\se

In this section, we work with an additional space discretizing parameter $H$, that corresponds to a coarse mesh. In other words, $0<h<H$ and both $h$ and $H$ tend to $0$. We introduce associated conforming finite element spaces $(\bH_H,L_H)$ and $(\bH_h,L_h)$ such that $(\bH_H,L_H) \subset (\bH_h,L_h)$. And this two-level finite element is to find a pair $(\bu^h,p^h)$ as follows:

{\bf First Level}: We compute the mixed finite element approximation $(\bu_H,p_H)\in (\bH_H,p_H)$ of $(\bu,p)$ of (\ref{wfh}). In other words, we solve the nonlinear problem in a coarse mesh. Find $(\bu_H,p_H)\in (H_H,p_H)$ satisfying
\begin{align}\label{2lvlH1}
(\bu_{Ht}, \bphi_H) +\nu a (\bu_H,\bphi_H)+& b(\bu_H,\bu_H,\bphi_H)-(p_H, \nabla \cdot \bphi_H) =(\f, \bphi_H), \nonumber \\
&(\nabla \cdot \bu_H, \chi_H) =0,
\end{align}
for $(\bphi_H, \chi_H)\in (\bH_H, L_H)$.

{\bf Second Level}: We solve a linearized problem on a fine mesh. In other words, we solve a Stokes problem. Find $(\bu^h,p^h) \in (\bH_h,L_h)$ satisfying
\begin{align}\label{2lvlH2}
(\bu^h_t, \bphi_h) +\nu a(\bu^h,\bphi_h)& -(p^h, \nabla \cdot \bphi_h) =(\f, \bphi_h)
- b(\bu_H,\bu_H,\bphi_h), \nonumber \\
&(\nabla \cdot \bu^h, \chi_h) =0,
\end{align}
for $(\bphi_h, \chi_h)\in (\bH_h, L_h)$.

\noindent An equivalent way is to look for solution in a weekly divergent free space.

{\bf First Level}: Find $\bu_H \in \bJ_H$ satisfying
\begin{align}\label{2lvlJ1}
(\bu_{Ht}, \bphi_H) +\nu a (\bu_H,\bphi_H)+ b(\bu_H,\bu_H,\bphi_H) =(\f, \bphi_H),
\end{align}
for $\bphi_H \in \bJ_H$.

{\bf Second Level}: With $\bu_H$ as the solution of (\ref{2lvlJ1}), find $\bu^h \in \bJ_h$ satisfying
\begin{align}\label{2lvlJ2}
(\bu^h_t, \bphi_h) +\nu a (\bu^h,\bphi_h) =(\f, \bphi_h)- b(\bu_H,\bu_H,\bphi_h),
\end{align}
for $\bphi_h \in \bJ_h$.

\begin{remark}
The well-posedness of the above systems can be seen from the facts that (\ref{2lvlH1}) or (\ref{2lvlJ1}) is classical Galerkin approximation and hence is well-posed as is stated in Section $3$. And (\ref{2lvlH2}) or (\ref{2lvlJ2}) represent linearized version. Given $\bu_H$ and with suitable $\bu^h(0)$, it is therefore an easy task to show the existence of an unique solution pair $(\bu^h,p^h)$ (or an unique solution $\bu^h$) for the linearized problem following the foot-steps of the non-linear problem.
\end{remark}
\noindent
Following Lemma \ref{est.uh}, we can easily obtain the {\it a priori} estimates of $\bu_H$.
\begin{lemma}\label{est.uH}
Under the assumptions of Lemma \ref{est.uh} and for $\bu_H(0)=P_H\bu_0$, the solution 
$\bu_H$ of (\ref{2lvlJ1}) satisfies, for $t>0,$
\begin{eqnarray}
\|\bu_H(t)\|+e^{-2\alpha t}\int_0^t e^{2\alpha s}\|\bu_H(t)\|_1^2~ds \le K, \label{uH01} \\
\|\bu_H(t)\|_1+e^{-2\alpha t}\int_0^t e^{2\alpha s}\|\bu_H(t)\|_2^2~ds \le K, \label{uH02} \\
(\tau^*(t))^{1/2}\|\bu_H(t)\|_2 \le K, \label{uH03}
\end{eqnarray}
where $\tau^*(t)= \min \{1,t\}$ and $K$ depends only on the given data. In particular, $K$ is independent of $H$ and $t$.
\end{lemma}
\noindent The following higher-order estimate of $\bu_H$ is required for error analysis. The proof of the same is similar to that of Lemma \ref{est1.uh}.
\begin{lemma}\label{est1.uH}
Under the assumptions of Lemma \ref{est.uH}, the solution $\bu_H$ of (\ref{2lvlJ1}) satisfies, for $t>0,$
\begin{equation}\label{1uH01}
(\tau^*)^2(t)\|\bu_{H,t}(t)\|_1^2+(\tau^*)^3(t)\|\bu_{H,t}(t)\|_2^2+e^{-2\alpha t}\int_0^t
\tau^*(s)e^{2\alpha s}\|\bu_{H,s}(s)\|^2~ds \le K.
\end{equation}
\end{lemma}
\noindent For the {\it a priori} estimates of $\bu^h$, we present another lemma.
\begin{lemma}\label{est.2uh}
Under the assumptions of Lemma \ref{est.uH}, the solution $\bu^h$ of (\ref{2lvlJ2}) satisfies, for $t>0,$
\begin{eqnarray}
\|\bu^h(t)\|+e^{-2\alpha t}\int_0^t e^{2\alpha s}\|\bu^h(t)\|_1^2~ds \le K, \label{2uh01} \\
\|\bu^h(t)\|_1+e^{-2\alpha t}\int_0^t e^{2\alpha s}\|\bu^h(t)\|_2^2~ds \le K, \label{2uh02}
\end{eqnarray}
\end{lemma}

\begin{proof}
Given $\bu_H$ along with the estimates of Lemma (\ref{est.uH}), we choose $\bphi_h= e^{2\alpha t}\bu^h(t)= e^{\alpha t} \hbu^h(t)$ in (\ref{2lvlJ2}) to obtain
\begin{align}\label{2uh001}
\frac{1}{2}\frac{d}{dt}\|\hbu^h\|^2-\alpha\|\hbu^h\|^2+\nu\|\hbu^h\|_1^2
=(\hf,\hbu^h)-e^{2\alpha t} b(\bu_H,\bu_H,\bu^h).
\end{align}
Use Cauchy-Schwarz inequality, Poincar\'e inequality with first eigenvalue of Stokes operator as the constant and Young's inequality, we have
\begin{align}\label{2uh002}
(\hf,\hbu^h) \le \|\hf\|\|\hbu^h\| \le \frac{1}{\lambda_1^{1/2}}\|\hf\|\|\hbu^h\|_1 
\le \frac{\nu}{4}\|\hbu^h\|_1^2+\frac{1}{\nu\lambda_1}\|\hf\|^2.
\end{align}
From Lemmas \ref{nonlin} and \ref{est.uH}, we find that
\begin{align*}
b(\bu_H,\bu_H,\bu^h) & \le \|\bu_H\|^{1/2}\|\bu_H\|_1^{3/2}\|\bu^h\|^{1/2}\|\bu^h\|_1^{1/2}
+\|\bu_H\|\|\bu_H\|_1\|\bu^h\|_1 \nonumber \\
& \le K+\frac{\nu}{4}\|\bu^h\|_1^2
\end{align*}
Therefore,
\begin{equation}\label{2uh003}
e^{2\alpha t} b(\bu_H,\bu_H,\bu^h) \le Ke^{2\alpha t}+\frac{\nu}{4}\|\hbu^h\|_1^2.
\end{equation}
Putting the estimates (\ref{2uh002})-(\ref{2uh003}) in (\ref{2uh001}) gives us
\begin{align*}
\frac{d}{dt}\|\hbu^h\|^2-2\alpha\|\hbu^h\|^2+\nu\|\hbu^h\|_1^2 \le Ke^{2\alpha t} +\frac{1}{\nu\lambda_1} \|\hf\|^2.
\end{align*}
Use Poincar\'e inequality to get
\begin{align}\label{2uh004}
\frac{d}{dt}\|\hbu^h\|^2+\big(\nu-\frac{2\alpha}{\lambda_1}\big)\|\hbu^h\|_1^2 \le Ke^{2\alpha t} +\frac{1}{\nu\lambda_1} \|\hf\|^2.
\end{align}
Since $0<\alpha< \nu\lambda_1/2$, we have that $\nu-2\alpha/\lambda_1>0$. \\
Integrate (\ref{2uh004}) with respect to time.
\begin{align*}
\|\hbu^h(t)\|^2+\big(\nu-\frac{2\alpha}{\lambda_1}\big)\int_0^t \|\hbu^h(s)\|_1^2 ds
\le Ke^{2\alpha t}.
\end{align*}
Multiply by $e^{-2\alpha t}$ to conclude (\ref{2uh01}). \\
For the next estimate, we choose $\bphi_h= e^{2\alpha t}\td_h\bu^h(t)$ in (\ref{2lvlJ2}) and proceed as above. For the non-linear term, we again follow similar proof technique of (\ref{nonlin03}) to obtain
\begin{align}\label{2uh005}
b(\bu_H,\bu_H,\td_h\bu^h) \le C\|\bu_H\|^{1/2}\|\bu_H\|_2^{1/2}
\|\bu_H\|_1\|\bu^h\|_2.
\end{align}
These estimates let us have
\begin{align}\label{2uh007}
\frac{1}{2}\frac{d}{dt}\|\hbu^h\|_1^2-\alpha\|\hbu^h\|_1^2+\nu\|\hbu^h\|_2^2 \le \|\hf\|\|\hbu^h\|_2+Ce^{\alpha t} \|\bu_H\|^{1/2}\|\bu_H\|_2^{1/2}\|\bu_H\|_1\|\hbu^h\|_2.
\end{align}
Using Cauchy-Schwarz inequality and Poincar\'e inequality as earlier, we find that
\begin{align}\label{2uh008}
\frac{d}{dt}\|\hbu^h\|_1^2+\big(\nu-\frac{\alpha}{2\lambda_1}\big)\|\hbu^h\|_2^2
\le C\|\hf\|^2+Ce^{2\alpha t}\big\{\|\bu_H\|^2\|\bu_H\|_1^4+\|\bu_H\|_2^2\big\}.
\end{align}
Integrate (\ref{2uh008}) with respect to time, use Lemma \ref{est.uH} and finally multiply by $e^{-2\alpha t}$ to obtain (\ref{2uh02}).
\end{proof}

\section{Error Estimate}
\se

In this section, we present the error estimate for the spatial approximation, that is, two-level finite element approximation. We achieve the desired results through a series of Lemmas. For the sake of convenience, henceforward, we will write $Ke^{Kt}$ simply as $K$.

\noindent
We denote the error, due to two-level method, as $\e=\bu_h-\bu^h$.

\noindent
From the equations (\ref{dwfj}) and (\ref{2lvlJ2}), we have the following error equation:
\begin{align}\label{err}
(\e_t,\bphi_h) +\nu a (\e,\bphi_h)= \Lambda_h(\bphi_h)~~\forall \bphi_h \in \bJ_h,
\end{align}
where
\begin{align}\label{lamb}
\Lambda_h(\bphi_h) & = b(\bu_H,\bu_H,\bphi_h)-b( \bu_h, \bu_h, \bphi_h) \nonumber \\
&= b(\bu_H,\bu_H-\bu_h,\bphi_h)+b(\bu_H-\bu_h,\bu_h,\bphi_h).
\end{align}

\begin{lemma}\label{neg} 
Under the assumptions of Theorem \ref{errest} and with the additional assumption that $\bu^h(0)=\bu_h(0)$ and for $0<\alpha< \nu\lambda_1/2$, the following estimate
\begin{equation}\label{neg1}
\|\e(t)\|_{-1}^2+e^{-2\alpha t}\int_0^t e^{2\alpha s}\|\e(s)\|^2 ds \le KH^4
\end{equation}
holds, for $t>0$.
\end{lemma}

\begin{proof}
Choose $\bphi_h=e^{2\alpha t}(-\td_h)^{-1}\e(t)= e^{\alpha t}(-\td_h)^{-1}\he(t)$ in (\ref{err}) to obtain
\begin{equation}\label{neg01}
\frac{1}{2}\frac{d}{dt}\|\he\|_{-1}^2-\alpha\|\he\|_{-1}^2+\nu\|\he\|^2= e^{2\alpha t}\Lambda_h((-\td_h)^{-1}\e).
\end{equation}
Following (\ref{nonlin03}) and using Lemmas \ref{est.uh} and \ref{est.uH}, we find
\begin{align*}
\Lambda((-\td_h)^{-1}\e) &= b(\bu_H,\bu_H-\bu_h,(-\td_h)^{-1}\e)+b(\bu_H-\bu_h,\bu_h,(-\td_h)^{-1}\e) \\
& \le \|\bu_H-\bu_h\|\|\e\|\big\{\|\bu_H\|_1+\|\bu_h\|_1\big\}
\le K\|\bu_H-\bu_h\|\|\e\|.
\end{align*}
Incorporate this in (\ref{neg01}).
\begin{equation}\label{neg02}
\frac{d}{dt}\|\he\|_{-1}^2+\nu\|\he\|^2 \le 2\alpha\|\he\|_{-1}^2+K\|\hbu_H-\hbu_h\|^2.
\end{equation}
Integrate (\ref{neg02}) with respect to time. Using (\ref{err.l2l2}), we first observe that
\begin{align}\label{est.hH}
\int_0^t \|(\hbu_H-\hbu_h)(s)\|^2 ds \le \int_0^t \|(\hbu-\hbu_H)(s)\|^2 ds
+\int_0^t \|(\hbu-\hbu_h)(s)\|^2 ds \le Ke^{2\alpha t}H^4.
\end{align}
And hence
\begin{equation}\label{neg03}
\|\he(t)\|_{-1}^2+\nu\int_0^t \|\he(s)\|^2 ds \le 2\alpha\int_0^t \|\he(s)\|_{-1}^2 ds +Ke^{2\alpha t}H^4.
\end{equation}
Multiply by $e^{-2\alpha t}$ and use Gronwall's Lemma to complete the rest of the proof.
\end{proof}

\begin{lemma}\label{pee} 
Under the assumptions of Lemma \ref{neg}, the following estimate
\begin{equation}\label{pee1}
\tau^*(t)\|\e(t)\|^2+e^{-2\alpha t}\int_0^t \sigma(s)\|\e(s)\|_1^2 ds \le K(t)H^4
\end{equation}
holds, for $t>0$, where $\sigma(t)=\tau^*(t)e^{2\alpha t},~\tau^*(t)=\min\{1,t\}$.
\end{lemma}

\begin{proof}
Choose $\bphi_h=\sigma(t)\e(t)$ in (\ref{err}) to obtain
\begin{align}\label{pee01}
\frac{d}{dt}\big\{\sigma(t)\|\e\|^2\big\}+2\nu~\sigma(t)\|\e\|_1^2= \sigma_t(t)\|\e\|^2
+2\sigma(t)\Lambda_h(\e).
\end{align}
To estimate the non-linear term, we follow (\ref{nonlin05}) and use Lemmas \ref{est.uH} and \ref{est.uh}.
\begin{align}\label{pee.nonlin}
\Lambda_h(\e) &\le C\|\bu_H-\bu_h\|\|\e\|_1\big\{\|\bu_H\|_1^{1/2}\|\bu_H\|_2^{1/2} +\|\bu_h\|_1^{1/2}\|\bu_h\|_2^{1/2}\big\} \nonumber \\
&\le K\|\bu_H-\bu_h\|\|\e\|_1\big\{\|\bu_H\|_2^{1/2}+\|\bu_h\|_2^{1/2}\big\}.
\end{align}
And hence, we find from (\ref{pee01})
\begin{align*}
\frac{d}{dt}\big\{\sigma(t)\|\e\|^2\big\}+2\nu~\sigma(t)\|\e\|_1^2 \le C\|\he\|^2
+K\sigma(t)\|\e\|_1\|\bu_H-\bu_h\|\big\{\|\bu_H\|_2^{1/2} +\|\bu_h\|_2^{1/2}\big\}.
\end{align*}
Use Cauchy-Schwarz inequality and then kickback argument to yield
\begin{align}\label{pee02}
\frac{d}{dt}\big\{\sigma(t)\|\e\|^2\big\}+\nu~\sigma(t)\|\e\|_1^2 \le C\|\he\|^2
+K\sigma(t)\|\bu_H-\bu_h\|^2\big\{\|\bu_H\|_2 +\|\bu_h\|_2\big\}.
\end{align}
Use Lemmas \ref{est.uH} and \ref{est.uh} to estimate the last term of (\ref{pee02}) as follows:
$$ K\sigma(t)\|\bu_H-\bu_h\|^2\big\{\|\bu_H\|_2 +\|\bu_h\|_2\big\}
 \le K\|\hbu_H-\hbu_h\|^2. $$
We have used the fact that $\tau^*(t) \le 1$. Incorporate this in (\ref{pee02}) and integrate in time. Use (\ref{est.hH}) to estimate the last term.
\begin{align}\label{pee03}
\sigma(t)\|\e(t)\|^2+\nu\int_0^t \sigma(s)\|\e(s)\|_1^2 ds \le C\int_0^t \|\he(s)\|^2 ds +Ke^{2\alpha t}H^4.
\end{align}
Now use Lemma \ref{neg} and then multiply the resulting inequality by $e^{-2\alpha t}$ to complete the rest of the proof.
\end{proof}

\begin{lemma}\label{eeu} 
Under the assumptions of Lemma \ref{neg}, the following estimate
\begin{equation}\label{eeu1}
(\tau^*(t))^2\|\e(t)\|_1^2+e^{-2\alpha t}\int_0^t \sigma_1(s)\|\e_s(s)\|^2 ds \le K(t)H^4
\end{equation}
holds, for $t>0$, where $\sigma_1(t)=(\tau^*(t))^2 e^{2\alpha t}$.
\end{lemma}

\begin{proof}
Choose $\bphi_h=\sigma_1(t)\e_t(t)$ in (\ref{err}) to obtain
\begin{align}\label{eeu01}
2\sigma_1(t)\|\e_t\|^2+\nu\frac{d}{dt}\big\{\sigma_1(t)\|\e\|_1^2\big\}= \nu\sigma_{1,t}(t)\|\e\|_1^2+2\sigma_1(t)\Lambda_h(\e_t).
\end{align}
Integrate (\ref{eeu01}) with respect to time. 
\begin{align}\label{eeu04}
2\int_0^t \sigma_1(s)\|\e_s(s)\|^2 ds +\nu\sigma_1(t)\|\e(t)\|_1^2 & \le
C\int_0^t \sigma(s)\|\e(s)\|_1^2 ds+2\int_0^t \sigma_1(s)\Lambda_h(\e_s)~ds.
\end{align}
We rewrite the non-linear terms as follows:
\begin{align*}
\Lambda_h(\e_t) &=b(\bu_H,\bu_H-\bu_h,\e_t)+b(\bu_H-\bu_h,\bu_h,\e_t) \\
&= \frac{d}{dt}\big\{b(\bu_H,\bu_H-\bu_h,\e)+b(\bu_H-\bu_h,\bu_h,\e)\big\}
-b(\bu_{H,t},\bu_H-\bu_h,\e) \\
&-b(\bu_H,\bu_{H,t}-\bu_{h,t},\e)-b(\bu_{H,t}-\bu_{h,t},\bu_h,\e)
-b(\bu_H-\bu_h,\bu_{h,t},\e).
\end{align*}
And hence
\begin{align}\label{eeu05}
\sigma_1(t)\Lambda_h(\e_t) &= \sigma_1(t)\big( b(\bu_H,\bu_H-\bu_h,\e_t) +b(\bu_H-\bu_h,\bu_h,\e_t)\big) \nonumber \\
&= \frac{d}{dt}\big\{\sigma_1(t)\big(b(\bu_H,\bu_H-\bu_h,\e)+b(\bu_H-\bu_h,\bu_h,\e)
\big)\big\} \nonumber \\
&-\sigma_{1,t}(t)\big(b(\bu_H,\bu_H-\bu_h,\e)+b(\bu_H-\bu_h,\bu_h,\e)
\big)-\sigma_1(t)b(\bu_{H,t},\bu_H-\bu_h,\e) \nonumber \\
&-\sigma_1(t)\big(b(\bu_H,\bu_{H,t}-\bu_{h,t},\e)+b(\bu_{H,t}-\bu_{h,t},\bu_h,\e)
+b(\bu_H-\bu_h,\bu_{h,t},\e)\big).
\end{align}
As seen earlier, we have
\begin{eqnarray*}
&&-b(\bu_{H,t},\bu_H-\bu_h,\e)-b(\bu_H-\bu_h,\bu_{h,t},\e) \\
& \le & C\|\bu_H-\bu_h\|\|\e\|_1\big\{\|\bu_{H,t}\|_1^{1/2}\|\bu_{H,t}\|_2^{1/2}
+\|\bu_{h,t}\|_1^{1/2}\|\bu_{h,t}\|_2^{1/2}\big\}
\end{eqnarray*}
Use Lemmas \ref{est1.uh} and \ref{est1.uH} to conclude that
\begin{equation}\label{eeu06}
-b(\bu_{H,t},\bu_H-\bu_h,\e)-b(\bu_H-\bu_h,\bu_{h,t},\e) \le K(\tau^*(t))^{-3/2} \|\bu_H-\bu_h\|\|\e\|_1.
\end{equation}
Similarly
\begin{equation}\label{eeu07}
-b(\bu_H,\bu_H-\bu_h,\e)-b(\bu_H-\bu_h,\bu_h,\e) \le K(\tau^*(t))^{-1/2}
\|\bu_H-\bu_h\|\|\e\|_1
\end{equation}
and
\begin{equation}\label{eeu08}
-b(\bu_H,\bu_{H,t}-\bu_{h,t},\e)-b(\bu_{H,t}-\bu_{h,t},\bu_h,\e) \le K(\tau^*(t))^{-1/2}
\|\bu_{H,t}-\bu_{h,t}\|\|\e\|_1.
\end{equation}
Incorporate (\ref{eeu06})-(\ref{eeu08}) in (\ref{eeu05}) and integrate with respect to time. Re-use (\ref{eeu07}) to find
\begin{align}\label{eeu09}
\int_0^t \sigma_1(s)\Lambda_h(\e_s) &~ds \le \sigma_1(t)\big(b(\bu_H,\bu_H-\bu_h,\e) +b(\bu_H-\bu_h,\bu_h,\e)\big) \nonumber \\
& +K\int_0^t e^{2\alpha s} \big\{(\tau^*(s))^{1/2} \|(\bu_H-\bu_h)(s)\|+(\tau^*(s))^{3/2}
\|(\bu_{H,s}-\bu_{h,s})(s)\|\big\}\|\e(s)\|_1 ds \nonumber \\
\le & Ke^{2\alpha t}(\tau^*(t))^{3/2}\|\bu_H-\bu_h\|\|\e\|_1+C\int_0^t \sigma(s)\|\e(s)\|_1^2 ds \nonumber \\
& +K\int_0^t e^{2\alpha s} \big\{\|(\bu_H-\bu_h)(s)\|^2+(\tau^*(s))^2\|(\bu_{H,s}-\bu_{h,s})(s)\|^2\big\}
~ds \nonumber \\
\le & \frac{\nu}{2}\sigma_1(t)\|\e\|_1^2+K\sigma(t)\|\bu_H-\bu_h\|^2+C\int_0^t \sigma(s) \|\e(s)\|_1^2 ds \nonumber \\
& +K\int_0^t e^{2\alpha s} \big\{\|(\bu_H-\bu_h)(s)\|^2+(\tau^*(s))^2\|(\bu_{H,s}-\bu_{h,s})(s)\|^2\big\}~ds.
\end{align}
Put (\ref{eeu09}) in (\ref{eeu04}) to obtain
\begin{align}\label{eeu10}
2\int_0^t \sigma_1(s)\|\e_s(s)\|^2 ds +\frac{\nu}{2}\sigma_1(t) &\|\e\|_1^2 \le
C\int_0^t \sigma(s)\|\e(s)\|_1^2 ds+K\sigma(t)\|\bu_H-\bu_h\|^2 \nonumber \\
& +K\int_0^t e^{2\alpha s}\big\{\|(\bu_H-\bu_h)(s)\|^2+(\tau^*(s))^2 \|(\bu_{H,s}-\bu_{h,s})(s)\|^2\big\}~ds.
\end{align}
Now, use Lemma \ref{pee}. For the remaining part, rewrite $\bu_H-\bu_h$ as $\bu-\bu_h-(\bu-\bu_H)$. Then use (\ref{err.l2l2}) and Theorems \ref{errest} and \ref{errest.t} to complete the rest of the proof.
\end{proof}

\begin{lemma}\label{prs} 
Under the assumptions of Lemma \ref{neg} and additionally that the assumption $({\bf B2}')$
holds, we have
\begin{equation}\label{prs1}
\|p_h-p^h\| \le K(t)(\tau^*(t))^{-1}H^2.
\end{equation}
\end{lemma}

\begin{proof}
The LBB condition $({\bf B2}')$ tells us that, for $t>0$
\begin{equation}\label{prs01}
\|p_h-p^h\|_{L^2/\R} \le K_0 \sup_{0\neq \bphi_h\in\bH_h} \frac{(p_h-p^h, \nabla\cdot\bphi_h)}{\|\bphi_h\|_1}.
\end{equation}
From (\ref{dwfh}) and (\ref{2lvlH2}), we have, for $\bphi_h\in \bH_h$
\begin{align*}
(p_h-p^h, \nabla \cdot \bphi_h) &=(\e_t,\bphi_h) +\nu a(\e,\bphi_h)-\Lambda_h(\bphi_h) \\
&\le C\|\bphi_h\|_1\big\{\|\e_t\|_{-1,h}+\|\e\|_1\big\}+K(\tau^*(t))^{-1/4} \|\bu_H-\bu_h\|\|\bphi_h\|_1.
\end{align*}
We have estimated $\Lambda_h$ as in (\ref{pee.nonlin}) and have used Lemmas \ref{est.uH} and \ref{est.2uh}.
Using (\ref{prs01}), we obtain, thanks to Lemma \ref{errest} and \ref{eeu}
\begin{align}\label{prs02}
\|p_h-p^h\|_{L^2/\R} \le C\|\e_t\|_{-1,h}+K(\tau^*(t))^{-1}H^2.
\end{align}
Following  (\ref{neg.h}), we observe
\begin{align*}
\|\e_t\|_{-1,h} \le \|\e_t\|_{-1}.
\end{align*}
Keeping in mind that, for $\bphi\in\bH_0^1, P_h\bphi\in\bJ_h$, we use (\ref{err}) to write $(\e_t,\bphi)$ as in (\ref{E.neg01}) and similar to estimates in (\ref{E.neg01a}), we find
\begin{align}\label{prs03}
(\e_t,\bphi) \le C\|\e\|_1\|\bphi\|_1+C\|\bu_H-\bu_h\|\{\|\bu_H\|_1^{1/2}\|\bu_H\|_2^{1/2}
+\|\bu_h\|_1^{1/2}\|\bu_h\|_2^{1/2}\}\|\bphi\|_1+Ch\|\e_t\|\|\bphi\|_1.
\end{align}
Use Lemmas \ref{est.uH}, \ref{est.2uh} and \ref{eeu}. We again rewrite $\bu_H-\bu_h$ as $\bu-\bu_h-(\bu-\bu_H)$ and use Theorem \ref{errest}. Now, for $\bphi\neq 0$, we divide (\ref{prs03}) by $\bphi$ to finally obtain
\begin{align*}
\|\e_t\|_{-1} \le K(\tau^*(t))^{-1}H^2+Ch\|\e_t\|.
\end{align*}
And hence, from (\ref{prs02}), we find
\begin{align*}
\|p_h-p^h\|_{L^2/\R} \le CH\|\e_t\|+K(\tau^*(t))^{-1}H^2.
\end{align*}
Now the following lemma completes the proof.
\end{proof}

\begin{lemma}\label{subop}
Under the assumptions of Lemma \ref{neg}, the following estimate holds, for $t>0.$
$$ \|\e_t\| \le K(\tau^*(t))^{-1}H. $$
\end{lemma}

\begin{proof}
Differentiate the error equation (\ref{err}) with respect to time.
\begin{align}\label{errt}
(\e_{tt},\bphi_h) +\nu a(\e_t,\bphi_h)= \Lambda_{h,t}(\bphi_h)~~\forall \bphi_h \in \bJ_h,
\end{align}
where
\begin{align}\label{lambt}
\Lambda_{h,t}(\bphi_h)=& b(\bu_{H,t},\bu_H-\bu_h,\bphi_h)+b(\bu_H,\bu_{H,t}-\bu_{h,t},
\bphi_h)+b(\bu_{H,t}-\bu_{h,t},\bu_h,\bphi_h) \nonumber \\
&+b(\bu_H-\bu_h,\bu_{h,t},\bphi_h).
\end{align}
Choose $\bphi_h=\sigma_1(t)\e_t$ in (\ref{errt}) to find
\begin{align}\label{subop01}
\frac{d}{dt}\big\{\sigma_1(t)\|\e_t\|^2\big\}+2\nu\sigma_1(t) \|\e_t\|_1^2 \le 
\sigma_{1,t}(t)\|\e_t\|^2+2\sigma_1(t)\Lambda_{h,t}(\e_t).
\end{align}
Similar to (\ref{eeu06})-(\ref{eeu08}), we obtain
\begin{align}\label{subop02}
2\sigma_1(t)\Lambda_{h,t}(\e_t) \le C\sigma_1(t)\|\e_t\|_1\big\{\|\bu_{H,t}\|_1 +\|\bu_{h,t}\|_1\big\}\|\bu_H-\bu_h\|_1 \nonumber \\
+C\sigma_1(t)\|\e_t\|_1\big\{\|\bu_H\|_1+\|\bu_h\|_1\big\}\|\bu_{H,t}-\bu_{h,t}\|_1.
\end{align}
Incorporate (\ref{subop02}) in (\ref{subop01}). Use kickback argument and then integrate with respect to time.
\begin{align}\label{subop02a}
\sigma_1(t)\|\e_t(t)\|^2+\nu\int_0^t \sigma_1(s)\|\e_t\|_1^2 \le C\int_0^t \sigma(s)
\|\e_t(s)\|^2 ds+K\int_0^t \sigma_1(s)\|(\bu_{H,s}-\bu_{h,s})(s)\|_1^2 ds \nonumber \\
+KH^2\int_0^t \sigma(s)\big(\|\bu_{H,s}(s)\|_1^2+\|\bu_{h,s}(s)\|_1^2\big)~ds
\end{align}
Use Lemmas \ref{est.u} and \ref{est.2uh} to bound the last term of (\ref{subop02a}). For the second last term, we rewrite $\bu_{H,s}-\bu_{h,s}$ as $\bu_s-\bu_{h,s}-(\bu_s-\bu_{H,s})$
and use Lemma \ref{low.est}.
\begin{equation}\label{subop03}
\sigma_1(t)\|\e_t(t)\|^2+\nu\int_0^t \sigma_1(s)\|\e_t\|_1^2 \le C\int_0^t \sigma(s)
\|\e_t(s)\|^2 ds+KH^2.
\end{equation}
To estimate the first term on the right hand-side of (\ref{subop03}), we put $\bphi_h=\sigma(t)\e_t$ in (\ref{err}).
\begin{align}\label{subop04}
\sigma(t)\|\e_t\|^2+\frac{\nu}{2}\frac{d}{dt}\{\sigma(t)\|\e\|_1^2\}=\sigma_t(t)\|\e\|_1^2
+\sigma(t)\Lambda_h(\e_t).
\end{align}
We take care of the non-linear term as earlier.
\begin{align*}
\Lambda_h(\e_t) &=b(\bu_H,\bu_H-\bu_h,\e_t)+b(\bu_H-\bu_h,\bu_h,\e_t) \\
&\le C\|\e_t\|\|\bu_H-\bu_h\|_1\big(\|\bu_H\|^{1/2}\|\bu_H\|_2^{1/2}+\|\bu_h\|^{1/2}
\|\bu_h\|_2^{1/2}\big) \\
&+C\|\e_t\|\|\bu_H-\bu_h\|^{1/2}\|\bu_H-\bu_h\|_1^{1/2}\big(\|\bu_H\|_1^{1/2}
\|\bu_H\|_2^{1/2}+\|\bu_h\|_1^{1/2}\|\bu_h\|_2^{1/2}\big) \\
&\le K\|\e_t\|\|\bu_H-\bu_h\|_1\big(\|\bu_H\|_2^{1/2}+\|\bu_h\|_2^{1/2}\big).
\end{align*}
We have used Lemmas \ref{est.uH} and \ref{est.uh}. Incorporate in (\ref{subop04}) and after kickback argument, integrate with respect to time.
\begin{align*}
\int_0^t \sigma(s)\|\e_s(s)\|^2 ds+\sigma(t)\|\e(t)\|_1^2 \le C\int_0^t e^{2\alpha s}
\|\e(s)\|_1^2 ds+K\int_0^t e^{2\alpha s}\|\E(s)\|_1^2 ds.
\end{align*}
Use (\ref{rem.05}) to obtain
\begin{align}\label{subop05}
\int_0^t \sigma(s)\|\e_s(s)\|^2 ds+\sigma(t)\|\e(t)\|_1^2 \le C\int_0^t e^{2\alpha s}
\|\e(s)\|_1^2 ds+KH^2.
\end{align}
To estimate the first term on the right hand-side of (\ref{subop05}), we choose $\bphi_h=
e^{2\alpha t}\e$ in (\ref{err}) to find
\begin{align*}
\frac{1}{2}\frac{d}{dt}\|\he\|^2+\big(\nu-\frac{\alpha}{\lambda_1}\big)\|\he\|_1^2 \le
e^{2\alpha t}\Lambda_h(\e) \le K\|\he\|_1\big\{e^{\alpha t}\|\bu_H-\bu_h\|_1\}.
\end{align*}
Use kickback argument and then integrate with respect to time.
\begin{align*}
e^{2\alpha t}\|\e(t)\|^2+\int_0^t e^{2\alpha s}\|\e(s)\|_1^2 ds \le K\int_0^t e^{2\alpha s}
\|\E(s)\|_1^2 ds.
\end{align*}
We recall that $\e(0)=0.$ Apply (\ref{rem.05}) to get
\begin{equation}\label{subop06}
e^{2\alpha t}\|\e(t)\|^2+\int_0^t e^{2\alpha s}\|\e(s)\|_1^2 ds \le KH^2.
\end{equation}
Now, from (\ref{subop05}), we have
\begin{align}\label{subop07}
\int_0^t \sigma(s)\|\e_s(s)\|^2 ds+\sigma(t)\|\e(t)\|_1^2 \le KH^2.
\end{align}
And from (\ref{subop03})
\begin{equation}\label{subop08}
\sigma_1(t)\|\e_t(t)\|^2+\nu\int_0^t \sigma_1(s)\|\e_t\|_1^2 \le KH^2.
\end{equation}
And this completes the rest of the proof.
\end{proof}

\begin{remark}
In the Lemma \ref{subop}, we have desired and obtained a lower order estimate, which is of $O(H)$ instead of $O(H^2)$. Note that in the proof, we have avoided $O(H^2)$ estimates. For example, for the second last term (\ref{subop02a}), we do not follow the argument that we have used to bound the last term of (\ref{eeu10}). Instead, we have used Lemma \ref{low.est}. An optimal estimate of $\|\e_t\|$ would read $O(t^{-3/2}H^2)$ and which could be achieved by choosing $\bphi_h=\sigma_2(t)\e_t$ in (\ref{errt}), instead of $\sigma_1(t)\e_t$ as in the proof of Lemma \ref{subop}. But this would lead to a higher order singularity of pressure error estimate at $t=0$. Our pressure error estimate (\ref{prs1}) and velocity error estimate (\ref{eeu1}) have same order of singularity at $t=0$, which is not the case in \cite{He04}. Even for smooth initial data, we can obtain error estimates for both velocity and pressure, that have same order of singularity at $t=0$, i.e, of $O(H^2)$, by following our proof technique. This improves the pressure error estimate of \cite{He04}. To be precise, in \cite{He04}, to estimate the pressure error, $\|\e_t\|$ is used (see ($5.39$) from \cite{He04}). Taking advantage of conforming finite element, we have used $\|\e_t\|_{-1}$ instead. And we are awarded with a better pressure error estimate.
\end{remark}

{\small  e-mails: deepjyotig@gmail.com and pddamazio@ufpr.br}

\end{document}